\documentclass[11pt]{article}
\usepackage{fullpage}

\usepackage{amsmath, amsfonts, amssymb, amsthm, amsbsy, amscd, bm, bbm}
\usepackage{array}
\usepackage[usenames]{color}
\usepackage[dvips]{graphicx}
\usepackage{subfigure}
\usepackage{algorithm, algorithmic}
\usepackage{url}
\usepackage{hyperref}

\def\Rp{\mathbb{R}_{+}}             

\def\F{\mathcal{F}}                 

\def\normaldist{\mathcal{N}}        

\def\div{\mathrm{div}}              
\def\cov{\mathbf{cov}}              
\def\SNR{\text{SNR}}                
\def\ev{\operatorname{\mathbb{E}}}  
\def\df{\mathrm{df}}                
\def\diff{\mathrm{d}}                 

\def\svt{\mathrm{SVT}}              
\def\bsvt{\mathrm{BSVT}}              
\def\svht{\mathrm{SVHT}}            
\def\sure{\textsc{SURE}}            
\def\ind{\mathbb{I}}                
\def\id{{\boldsymbol{I}}}           
\def\idx{\mathcal{I}}               
\def\real{\mathrm{Re}}              
\def\imag{\mathrm{Im}}              
\def\1i{\imath}                     
\def\iid{\textrm{iid}}                

\newcommand{\vc}[1]{{\boldsymbol{#1}}}              
\newcommand{\conj}[1]{\overline{#1}}                
\renewcommand{\exp}[1]{\ensuremath{\text{e}^{#1}}}

\newcommand{\iprod}[2]{\left\langle #1 , #2 \right\rangle}

\newtheorem{theorem}{Theorem}[section]
\newtheorem{lemma}[theorem]{Lemma}
\newtheorem{corollary}[theorem]{Corollary}
\newtheorem{proposition}[theorem]{Proposition}

\newcommand{\R}{\mathbb{R}}
\newcommand{\C}{\mathbb{C}}

\newcommand{\rank}{\operatorname{rank}}

\newcommand{\diag}{\operatorname{diag}}

\def \endprf{\hfill {\vrule height6pt width6pt depth0pt}\medskip}

\renewenvironment{proof}{\noindent {\bf Proof} }{\endprf\par}

\definecolor{eac}{RGB}{200,50,50}
\definecolor{csl}{RGB}{50,200,50}
\definecolor{ejc}{RGB}{255,0,0}
\definecolor{jdt}{RGB}{0,0,255}

\numberwithin{equation}{section}

\allowdisplaybreaks

\title{Unbiased Risk Estimates for Singular Value
  Thresholding\\ and Spectral Estimators
}

\author{
Emmanuel J. Cand\`es\thanks{Departments of Statistics and of
    Mathematics, Stanford University, Stanford, CA 94305}
\and Carlos A. Sing-Long \thanks{Institute for Computational and Mathematical
      Engineering, Stanford University, Stanford, CA 94305}
\and Joshua D. Trzasko \thanks{Department of Physiology and Biomedical Engineering,
 Mayo Clinic, Rochester, MN 55905}
}

\date{October, 2012}

\begin{document}

%
\maketitle

\vspace{-0.3in}

\begin{abstract}
  In an increasing number of applications, it is of interest to
  recover an approximately low-rank data matrix from noisy
  observations. This paper develops an unbiased risk
  estimate---holding in a Gaussian model---for any {\em spectral
    estimator} obeying some mild regularity assumptions. In
  particular, we give an unbiased risk estimate formula for singular
  value thresholding (SVT), a popular estimation strategy which
  applies a soft-thresholding rule to the singular values of the noisy
  observations. Among other things, our formulas offer a principled
  and automated way of selecting regularization parameters in a
  variety of problems.  In particular, we demonstrate the utility of
  the unbiased risk estimation for SVT-based denoising of real
  clinical cardiac MRI series data. We also give new results
  concerning the differentiability of certain matrix-valued functions.
\end{abstract}




\section{Introduction}
\label{section:introduction}

\newcommand{\ourstd}{\tau}

\newcommand{\true}{\vc{X}^0}

Suppose we have noisy observations $\vc{Y}$ about an $m \times n$ data
matrix $\true$ of interest,
\begin{equation}\label{eq:mvn_mean_estimation}
  Y_{ij} = X^0_{ij} +  W_{ij},\qquad W_{ij} \stackrel{\text{iid}}{\sim}
\normaldist(0, \ourstd^2),\quad \begin{array}{l} i = 1,\ldots,m,\\
  j  = 1,\ldots, n.
\end{array}
\end{equation}
We wish to estimate $\true$ as accurately as possible. In this paper,
we are concerned with situations where the estimand has some
structure, namely, $\true$ has low rank or is well approximated by a
low-rank matrix. This assumption is often met in practice since the
columns of $\true$ can be quite correlated. For instance, these
columns may be individual frames in a video sequence, which are
typically highly correlated. Another example concerns the acquisition
of hyperspectral images in which each column of $\true$ is a 2D image
at a given wavelength. In such settings, images at nearby wavelengths
typically exhibit strong correlations. Hence, the special low-rank
regression problem \eqref{eq:mvn_mean_estimation} occurs in very many
applications and is the object of numerous recent studies.

Recently, promoting low-rank has been identified as a promising
tool for denoising series of MR images, such as those that arise in
functional MRI (fMRI)~\cite{Thomas2002}, relaxometry~\cite{Bydder2006},
cardiac MRI~\cite{Trzasko2011}, NMR spectroscopy~\cite{Cadzow1988,Nguyen2011}, and
diffusion-weighted imaging~\cite{Lam2012,Manjon2012}, among
others. 
In dynamic applications like cine cardiac
imaging, where ``movies'' of the beating heart are created,
neighboring tissues tend to exhibit similar motion profiles through
time due to their physical connectedness.  The diversity of temporal
behaviors in this settings will, by nature, be limited, and the
Casorati matrix (a matrix whose columns comprise vectorized frames of
the image series) formed from this data will be low-rank~\cite{Liang2007}. For example, in breath-hold cine cardiac
imaging, background tissue is essentially static and the large
submatrix of the Casoratian corresponding to this region is
very well approximated by a rank-1 matrix.

\subsection{Singular value thresholding}

Whenever the object of interest has (approximately) low rank, it is
possible to improve upon the naive estimate $\hat{\vc{X}}_0 = \vc{Y}$
by regularizing the maximum likelihood. A natural approach consists in
truncating the singular value decomposition of the observed matrix
$\vc{Y}$, and solve
\begin{equation}
\label{eq:hard}
\svht_\lambda(\vc{Y}) = \arg\min_{\vc{X}\in\R^{m\times n}} \frac{1}{2}\|\vc{Y} - \vc{X}\|_F^2 + \lambda\, \rank(\vc{X}),
\end{equation}
where $\lambda$ a positive scalar.  As is well known, if
\begin{equation}
  \label{eq:svd}
  \vc{Y} = \vc{U} \vc{\Sigma} \vc{V}^* = \sum_{i = 1}^{\min(m, n)}
  \sigma_i \vc{u}_i \vc{v}_i^*
\end{equation}
is a singular value decomposition for $\vc{Y}$, the solution is given
by retaining only the part of the expansion with singular values
exceeding $\lambda$,
\[
\svht_\lambda(\vc{Y}) = \sum_{i = 1}^{\min(m, n)}\ind(\sigma_i > \lambda) \,
\vc{u}_i \vc{v}_i^*.
\]
In other words, one applies a
hard-thresholding rule to the singular values of the observed matrix
$\vc{Y}$. Such an estimator is discontinuous in $\vc{Y}$ and a popular
alternative approach applies, instead, a soft-thresholding rule to the
singular values:
\begin{equation}
  \label{eq:svt}
 \svt_\lambda(\vc{Y}) = \sum_{i = 1}^{\min(m, n)} (\sigma_i - \lambda)_+ \,
\vc{u}_i \vc{v}_i^*;
\end{equation}
that is, we shrink the singular values towards zero by a constant
amount $\lambda$. Here, the estimate $\svt_\lambda(\vc{Y})$ is
Lipschitz continuous. This follows from the fact that the singular
value thresholding operation \eqref{eq:svt} is the prox of the nuclear
norm $\|\cdot\|_*$ (the nuclear norm of a matrix is sum of its
singular values), i.e.~is the unique solution to
\begin{equation}\label{eq:svt_as_prox}
  \text{min} \quad  \frac{1}{2}\|\vc{Y} - \vc{X}\|_F^2 + \lambda \|\vc{X}\|_{*}.
\end{equation}
The focus of this paper is on this smoother estimator referred to as
the singular value thresholding ($\svt$) estimate.

\subsection{A SURE formula}

The classical question is, of course, how much shrinkage should be
applied. Too much shrinkage results in a large bias while too little
results in a high variance. To find the correct trade-off, it would be
desirable to have a method that would allow us to compare the quality
of estimation for different values of the parameter
$\lambda$. Ideally, we would like to select $\lambda$ as to minimize
the mean-squared error or risk
\begin{equation}\label{eq:risk_mse}
  \text{MSE}(\lambda) = \ev\|\true - \svt_\lambda(\vc{Y})\|_F^2.
\end{equation}
Unfortunately, this cannot be achieved since the expectation in
\eqref{eq:risk_mse} depends on the true $\true$, and is thus
unknown.
Luckily, when the observations follow the model
\eqref{eq:mvn_mean_estimation},
it is possible to construct an {\it unbiased} estimate of the risk,
namely, Stein's Unbiased Risk Estimate (SURE)~\cite{Stein81} given by
\begin{equation}\label{eq:sure_for_svt}
  \sure(\svt_\lambda)(\vc{Y}) = - mn\ourstd^2 + \sum_{i = 1}^{\min(m, n)}\min(\lambda^2, \sigma_i^2) + 2\ourstd^2\div\, (\svt_\lambda(\vc{Y})),
\end{equation}
where $\{\sigma_i\}_{i = 1}^{n}$ denote the singular values of
$\vc{Y}$. Above, `$\div$' is the divergence of the nonlinear mapping
$\svt_\lambda$, which is to be interpreted in a weak sense. Roughly
speaking, it can fail to exist on negligible sets. The main
contribution of this paper is to provide a closed-form expression for
the divergence of this estimator.
We prove that in the
real-valued case,
\begin{equation}\label{eq:weak_div_for_sure}
  \div\, (\svt_\lambda (\vc{Y})) = \sum_{i = 1}^{\min(m,n)}\left[\ind(\sigma_i > \lambda) + |m-n|\left(1 - \frac{\lambda}{\sigma_i}\right)_+\right] + 2\sum_{i \neq j, i,j = 1}^{\min(m,n)} \frac{\sigma_i(\sigma_i - \lambda)_+}{\sigma_i^2 - \sigma_j^2},
\end{equation}
when $\vc{Y}$ is {\it simple}---i.e., has no repeated singular
values---and $0$ otherwise, say, is a valid expression for the weak
divergence. Hence, this simple formula can be used
in~\eqref{eq:sure_for_svt}, and ultimately leads to the determination
of a suitable threshold level by minimizing the {\it estimate} of the
risk, which only depends upon the observed data.

We pause to present some numerical experiments to assess the quality
of SURE as an estimate of the risk.
We work with four matrices $\true_i$, $ i = 1,\ldots, 4$, of size
$200\times 500$. Here, $\true_1$ has full rank; $\true_2$ has rank
$100$; $\true_3$ has rank $10$; and $\true_4$ has singular values
equal to $\sigma_i = \sqrt{200}/(1+\exp{(i - 100)/20})$, $i =
1,\ldots, 200$. Finally, each matrix is normalized so that
$\|\true_i\|_F = 1$, $i = 1,\ldots, 4$. Next, two methods are used to
estimate the risk~\eqref{eq:risk_mse} of $\svt_\lambda$ seen as a
function of $\lambda$. The
first methods uses
\begin{equation*}
    \hat{R}_i(\lambda) = \frac{1}{50}\sum_{j = 1}^{50} \|\svt_\lambda(\vc{Y}_j^{(i)}) - \true_i\|_F^2,
\end{equation*}
where $\{\vc{Y}^{(j)}_{i}\}_{j = 1}^{50}$ are independent samples
drawn from model~\eqref{eq:mvn_mean_estimation} with $\true =
\true_i$. We set this to be the value of reference. The second uses
$\sure(\svt_\lambda)(\vc{Y})$, where $\vc{Y}$ is drawn from
model~\eqref{eq:mvn_mean_estimation} independently from
$\{\vc{Y}^{(i)}_{j}\}$. Finally, in each case we work with values of
the signal-to-noise ratio, defined here as $\SNR =
\|\true\|_F/\sqrt{mn}\ourstd = 1/\sqrt{mn}\ourstd$, and set $\SNR =
0.5, 1, 2$ and $4$.  The results are shown in
Figure~\ref{figure:sure_vs_risk}, where one can see that SURE remains
very close to the true value of the risk, even though it is calculated
from a single observation.  Matlab code reproducing the figures is
available and computing SURE formulas for various spectral estimators
is available at
\url{http://www-stat.stanford.edu/~candes/SURE_SVT_code.zip}.
\begin{figure}
    \centering
    \subfigure[]{\includegraphics[width=0.45\textwidth]{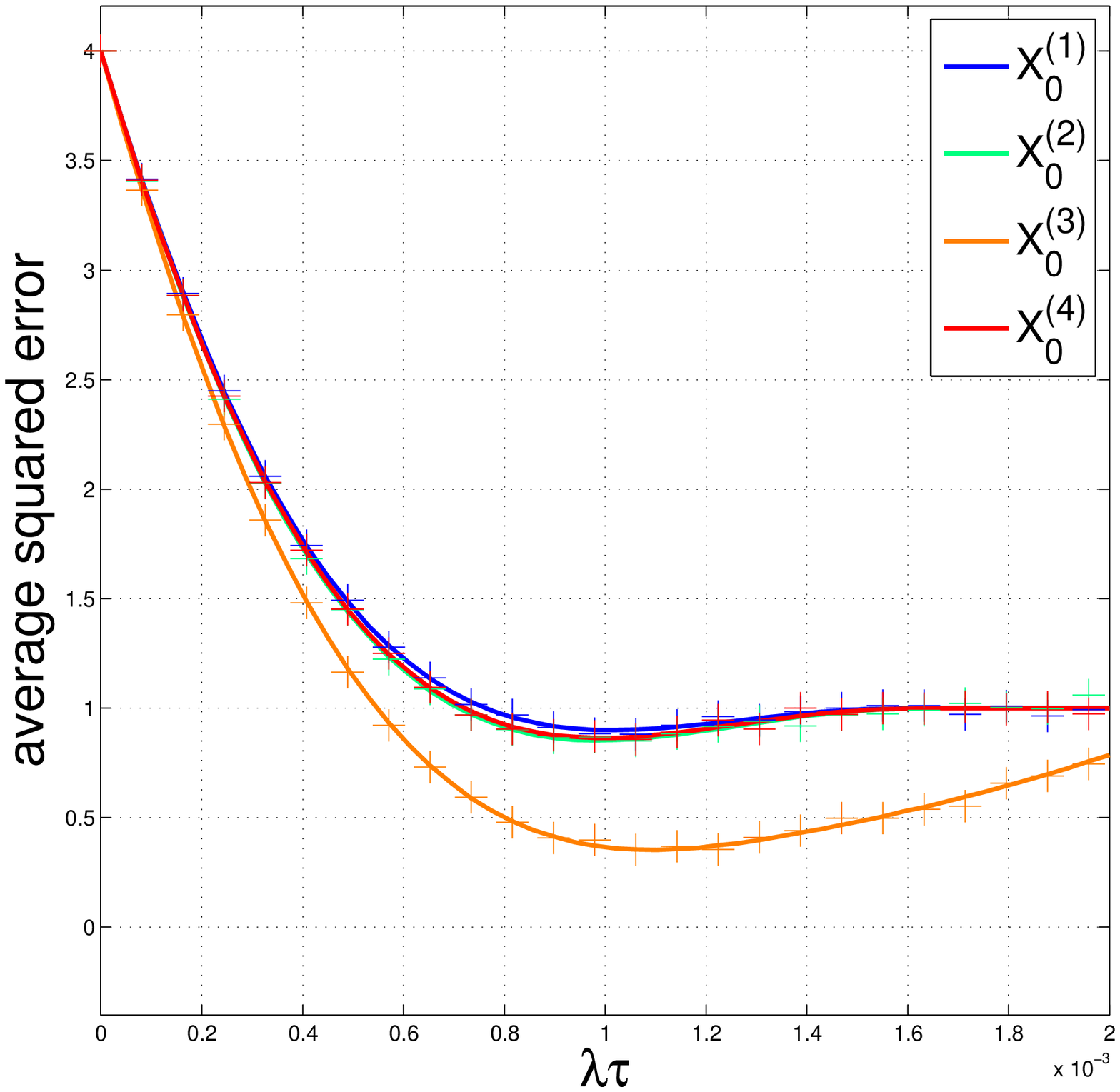}}
    \subfigure[]{\includegraphics[width=0.45\textwidth]{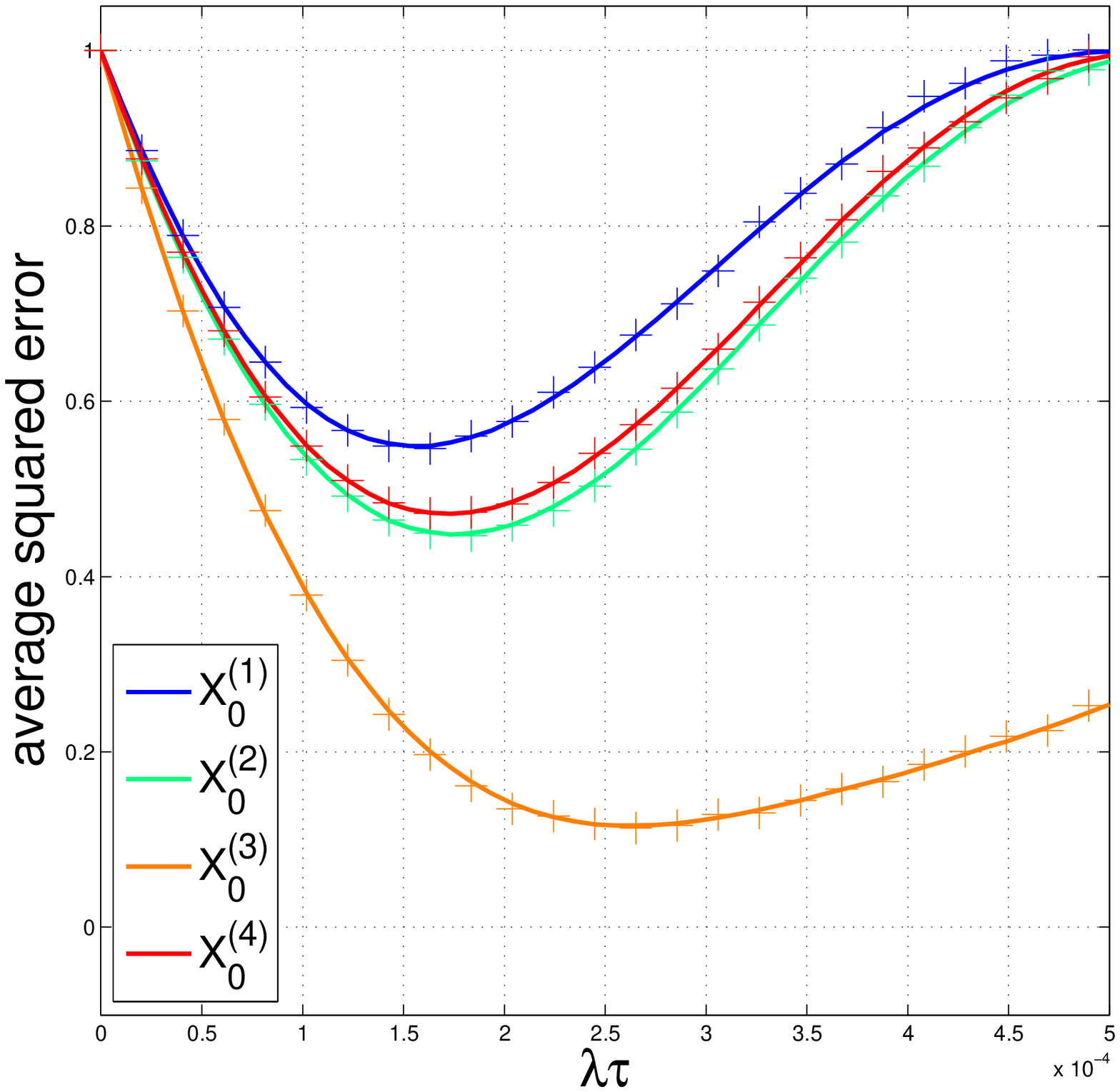}}
    \vfill
    \subfigure[]{\includegraphics[width=0.45\textwidth]{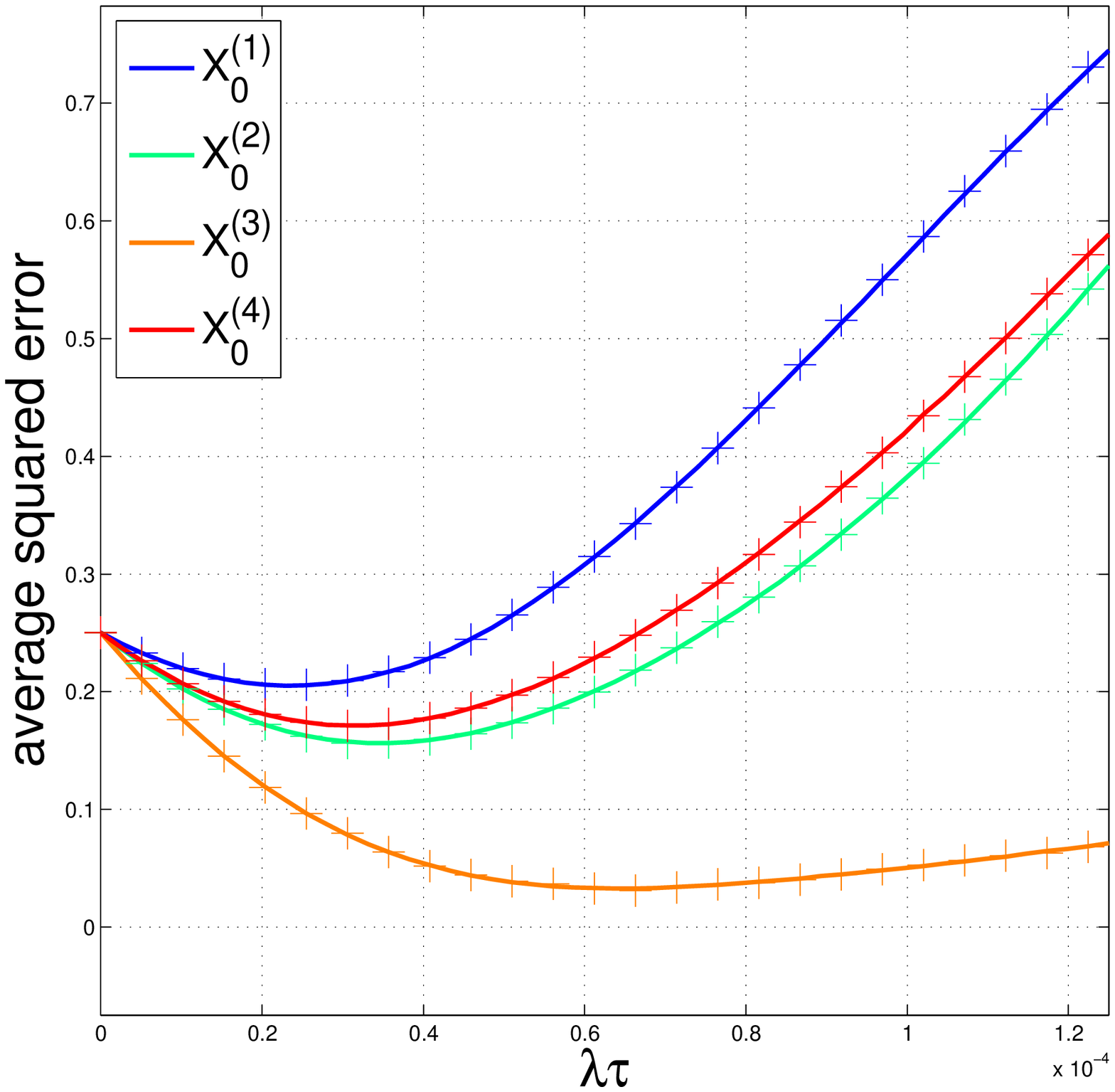}}
    \subfigure[]{\includegraphics[width=0.45\textwidth]{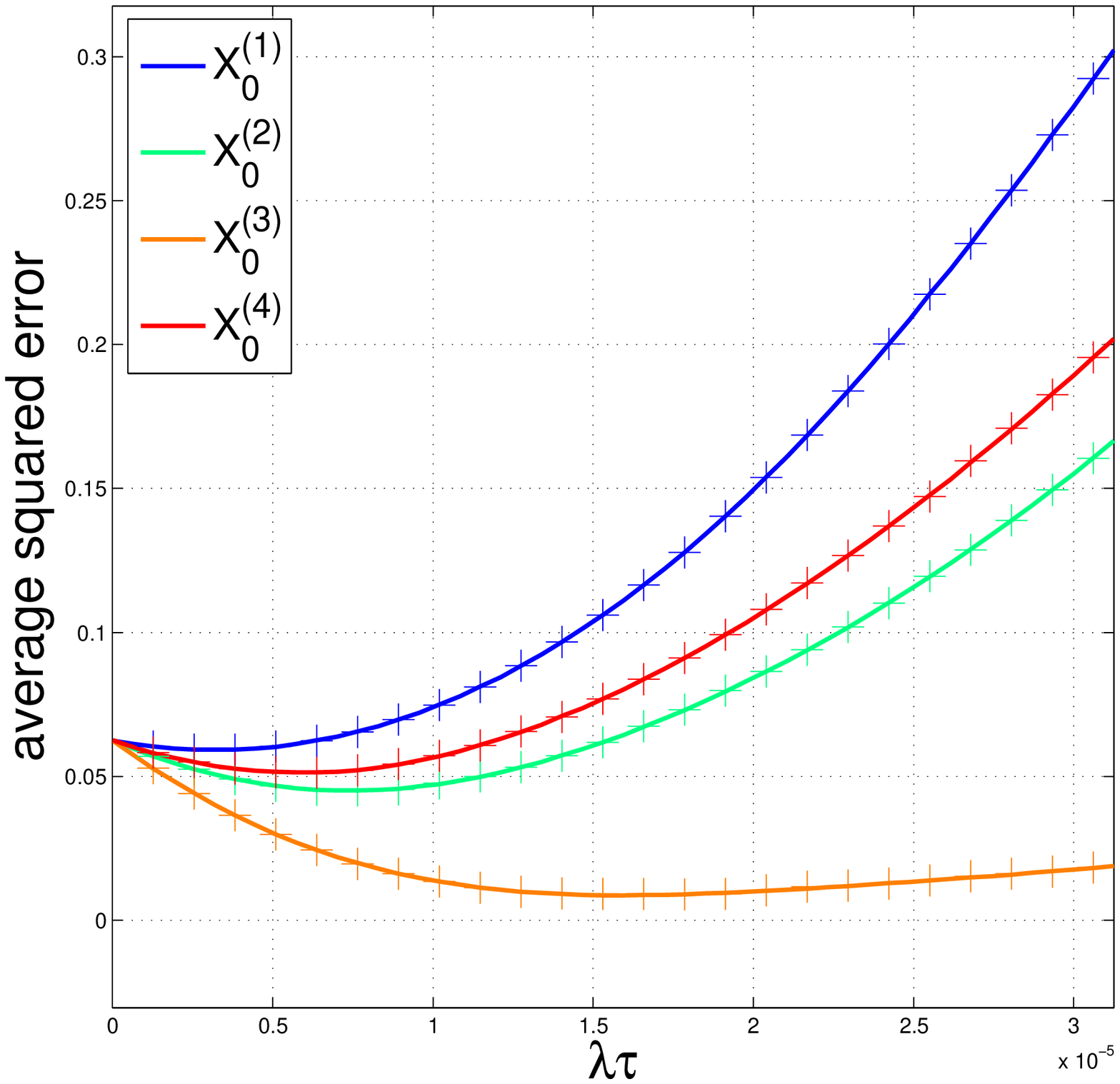}}
    \caption{Comparison of the risk estimate using Monte Carlo (solid
      line) and SURE (cross) versus $\lambda\times \tau$ for
      $\true\in\R^{200\times 500}$ and: (a) $\SNR = 0.5$; (b)
      $\SNR = 1$; (c) $\SNR = 2$; and (d) $\SNR = 4$. The colors
      indicate the matrices used in each case. As we can see, SURE
      follows closely the Monte Carlo estimates, even though it
      requires only one observation to estimate the risk. }
\label{figure:sure_vs_risk}
\end{figure}

In MR applications, observations can take on complex values.  Indeed,
MRI scanners employ a process known as quadrature detection (two
orthogonal phase-sensitive detectors) to observe both the magnitude
and phase of rotating magnetization that is induced by radio-frequency
(RF) excitation~\cite{Liang2000}. Quadrature detection both increases
SNR and allows for the encoding of motion information like flow. MRI
noise, which is Gaussian, is also complex valued. In general, MRI
noise can also be assumed \iid, noting that inter-channel correlation
in parallel MRI data can be removed via Cholesky
pre-whitening. Thus, model~\eqref{eq:mvn_mean_estimation} has to be modified as
\begin{equation}\label{eq:mvn_mean_estimation_cpx}
  \vc{Y} = \true + \vc{W},\qquad \real(W_{ij}),\,\imag(W_{ij}) \stackrel{\text{iid}}{\sim} \normaldist(0, \ourstd^2), 
\end{equation}
where the real and imaginary parts are also independent. In this case,
SURE becomes
\begin{equation*}
  \sure(\svt_\lambda)(\vc{Y}) = -2mn\ourstd^2 + \sum_{i = 1}^{\min(m, n)}\min(\lambda^2, \sigma_i^2) + 2\ourstd^2\div\, (\svt_\lambda(\vc{Y})).
\end{equation*}
We also provide an expression for the weak divergence in this context, namely,
\begin{equation}\label{eq:weak_div_for_sure_cpx}
  \div\, (\svt_\lambda(\vc{Y})) = \sum_{i = 1}^{\min(m,n)}\left[\ind(\sigma_i > \lambda) + (2|m-n| + 1)\left(1 - \frac{\lambda}{\sigma_i}\right)_+\right]+ 4\sum_{i\neq j,\, i, j =
    1}^{\min(m, n)} \frac{\sigma_i (\sigma_i - \lambda)_+}{\sigma_i^2 - \sigma_j^2},
\end{equation}
when $\vc{Y}$ is simple, and 0 otherwise. This formula can be readily
applied to MR data as we shall see in Section~\ref{section:mri_applications}. The decomposition into real and
imaginary parts would suggest that the divergence would be
proportional to twice the divergence in the real case. However, this
is not the case. The most significant difference is that there is a
contribution of the inverse of the singular values even when the
matrix is square.

\subsection{Extensions}

The formulae
\eqref{eq:weak_div_for_sure}--\eqref{eq:weak_div_for_sure_cpx} have
applications beyond SURE. For instance, the divergence of an
estimation procedure arises naturally when trying to estimate the
degrees of freedom associated with it.
Consider the estimate $\svt_\lambda(\vc{Y})$ of $\true$ when the observation
$\vc{Y}$ comes from an additive error model. The degrees of
freedom~\cite{Efron04,Stein81} are defined as
\begin{equation*}
  \df(\svt_\lambda) = \sum_{i = 1}^{m}\sum_{j = 1}^{n}\cov\left(\svt_{\lambda}(\vc{Y})_{ij}, Y_{ij}\right).
\end{equation*}
When the observations $\vc{Y}$ follow model
\eqref{eq:mvn_mean_estimation}, it is possible to write the degrees of
freedom as
\begin{equation*}
    \df(\svt_\lambda) = \ev\{\div(\svt_\lambda(\vc{Y}))\}.
\end{equation*}
Therefore, the expression we provide is also useful to estimate or
calculate the degrees of freedom of singular value
thresholding. 

Finally, although our work focuses on $\svt$, our methods extend
beyond this particular estimator. Consider estimators given by {\it
  spectral functions}.  These act on the singular values and take the
form
\begin{equation}\label{eq:spectral_function}
f(\vc{Y}) = \sum_{i = 1}^{\min(m, n)} f_i(\sigma_i) \vc{u}_i \vc{v}_i^* := \vc{U} f(\vc{\Sigma}) \vc{V}^*, \quad \text{for all } \vc{Y}\in\R^{m\times n},
\end{equation}
where $\vc{Y} = \vc{U} \vc{\Sigma} \vc{V}^*$ is any SVD ($\svt$ is in
this class).  Our methods show that
these functions admit a SURE formula, given by
\[
\sure(f)(\vc{Y}) = -mn\ourstd^2 + \|f(\vc{Y}) - \vc{Y}\|_F^2 +
2\ourstd^2\div\, (f(\vc{Y})),
\]
and that under mild assumptions there exists a closed-form for their
divergence:
\begin{equation}
\label{eq:gen-div}
    \div\, (f(\vc{X})) = \sum_{i = 1}^{\min(m, n)} \left(f'_i(\sigma_i) + |m-n|\frac{f'_i(\sigma_i)}{\sigma_i}\right)+ 2\sum_{i\neq j,\, i, j = 1}^{\min(m, n)} \frac{\sigma_i f_i(\sigma_i)}{\sigma_i^2 - \sigma_j^2}.
  \end{equation}
  This is of interest because such estimators arise naturally in
  regularized regression problems. For instance, let $J:\R^{m\times
    n}\mapsto \R$ be a lower semi-continuous, proper convex function
  of the form
\begin{equation*}
  J(\vc{X}) = \sum_{i = 1}^{\min(m, n)} J_i(\sigma_i(\vc{X})),
\end{equation*}
Then, for $\lambda > 0$ the estimator
\begin{equation}\label{eq:spectral_function_prox}
  f_\lambda(\vc{Y}) = \arg\min_{\vc{X}\in\R^{m\times n}} \frac{1}{2}\|\vc{Y} - \vc{X}\|_F^2 + \lambda J(\vc{X})
\end{equation}
is spectral. Hence, \eqref{eq:spectral_function_prox} can be used
broadly.

\subsection{Connections to other works}

During the preparation of this manuscript, the conference paper
\cite{Deledalle12} came to our attention and we would like to point
out some differences between this work and
ours. In~\cite{Deledalle12}, the authors propose to recursively {\it
  estimate} the divergence of an estimator given
by~\eqref{eq:spectral_function_prox}. This is done by using proximal
splitting algorithms. To this end, they provide an expression for the
directional subdifferential of a matrix-valued spectral function. The
divergence is then estimated by averaging subdifferentials taken along
random directions at each iteration of the proximal algorithm. Our
approach is obviously different since we provide closed-form formulas
for the divergence that are simple to manipulate and easy to
evaluate. This has clear numerical and conceptual advantages. In
particular, since we have a closed-form expression for SURE, it
becomes easier to understand the risk of a family of
estimators. Finally, we also address the case of complex-valued data
that seems out of the scope of~\cite{Deledalle12}. 

\subsection{Content}

The structure of the paper is as follows. In
Section~\ref{section:mri_applications} we present applications in MRI
illustrating the advantages of choosing the threshold in a disciplined
fashion. Section~\ref{section:sure_for_svt} provides precise
statements, and a rigorous justification for
\eqref{eq:weak_div_for_sure} and \eqref{eq:weak_div_for_sure_cpx}.
Section~\ref{section:general_theory} deals with the differentiability
of spectral functions~\eqref{eq:spectral_function}, and thus supports
Section~\ref{section:sure_for_svt}. We conclude with a short
discussion of our results, and potential research directions in
Section~\ref{section:discussion}.

\section{Applications in Magnetic Resonance Imaging (MRI)}
\label{section:mri_applications}

SVT is a computationally-straightforward, yet, powerful denoising
strategy for MRI applications where spatio-temporal/parameteric
behavior is either {\it a priori} unknown or else defined accordingly
to a complicated nonlinear model that may be numerically challenging
to work with. A practical challenge in most rank-based estimation
problems in MRI (and inverse problems in general), however, lies in
the selection of the regularization or threshold parameter. Defining
an automated, disciplined, and consistent methodology for selecting
this parameter inherently improves clinical workflow both by
accelerating the tuning process through the use of optimization-based,
rather than heuristic, search strategies and by freeing the MRI
scanner technician so that they can focus on other patient-specific
tasks. Moreover, eliminating the human element from the denoising
process mitigates inter- and intra-operator variability, and thus
raises diagnostic confidence in the denoising results since images are
wholly reproducible.

\subsection{Noise reduction in cardiac MRI series}

Here, we demonstrate a potential use of the SVT unbiased risk estimate
for automated and optimized denoising of dynamic cardiac MRI series.
Dynamic cardiac imaging is performed either in cine or real-time
mode. Cine MRI, the clinical gold-standard for measuring cardiac
function/volumetrics~\cite{Bogaert2005}, produces a movie of roughly 20 cardiac phases over a single cardiac
cycle (heart beat). However, by exploiting the semi-periodic nature
of cardiac motion, it is actually formed over many heart beats. Cine
sampling is gated to a patient's heart beat, and as each data
measurement is captured it is associated with a particular cardiac
phase. This process continues until enough data has been collected
such that all image frames are complete. Typically, an entire 2D cine
cardiac MRI series is acquired within a single breath hold (less than
30 secs).

In real-time cardiac MRI, an image series covering many heart beats is
generated. Although the perceived temporal resolution of this series
is coarser than that of a cine series, the ``temporal footprint''
of each frame is actually shorter since it is only comprised of
data from one cardiac cycle. Real-time cardiac MRI is of increasing
clinical interest for studying transient functional processes such as
first-pass myocardial perfusion, where the hemodynamics of an
intravenously-injected contrast bolus are visualized. First-pass
myocardial perfusion proffers visualization of both damage to heart
muscle as well as coronary artery blockage. Typically, a real-time
cardiac MRI study will exceed feasible breath-hold time and
respiratory motion may be visible.

Simultaneously achieving high spatial resolution and SNR is
challenging in both cine and real-time cardiac MRI due to the dynamic
nature of the target signal. Signal averaging (NEX~$\geq 1$), a
standard MRI techniques for noise reduction, is infeasible in
real-time imaging and undesirable in cine imaging due to potential
misregistration artifacts, since it cannot be executed within a single
breath-hold. Real-time acquisitions, which are generally based on
gradient recalled echo (GRE) protocols, have inherently low SNR due to
their use of very short repetition (TR) and echo times (TE). Cine
acquisitions use either standard GRE or balanced steady-state free
precession (bSSFP) protocols. When imaging with 1.5 T (Tesla)
magnets, bSSFP cine sequences can often yield sufficient SNR.
However, unlike many other MRI protocols, SSFP sequences do not
trivially gain SNR when moved to higher-field systems ($\geq$ 3.0 T)~\cite{Oshinski2010} which is an emerging clinical trend. As magnetic field
strength is increased, the RF excitation flip angle used in a bSSFP
sequence must be lowered to adhere to RF power deposition (SAR) safety
limits. This results in weaker signal excitation which can mitigate
gains in bulk magnetization. Poor receiver coil sensitivity at the
heart's medial location and signal loss due to iron
overload (hemochromatosis), among other factors, can further
reduce SNR in both imaging strategies. Beyond confounding visual
radiological evaluation, even moderate noise levels in a cardiac image
can also degrade the performance of automated segmentation methods
that are used for quantitative cardiac function evaluation. Therefore, effective denoising techniques that
preserve the morphological and dynamic profiles of cardiac image
series are clinical valuable.

Consider a series of $t$ separate $n\times{n}$ 2D MR images. To
utilize SVT to denoise this data, it must first be transformed into a
$n^{2}\times{t}$ Casorati matrix, $\vc{Y} = \true + \vc{W}$.
In many MRI scenarios, spatial
dimensionality will greatly exceed temporal/parametric dimensionality
($n^2\gg t$) and $\vc{Y}$ typically is a very thin matrix. Due to
limited degrees-of-freedom, even optimally-parameterized SVT may result
in temporal/parametric blurring. One solution to this problem is to
analyze the image series in a spatial block-wise manner (see Figure~\ref{figure:local_svt}) rather than globally~\cite{Trzasko2011}.

\begin{figure}[ht!]
\centering
\includegraphics[width=0.45\textwidth]{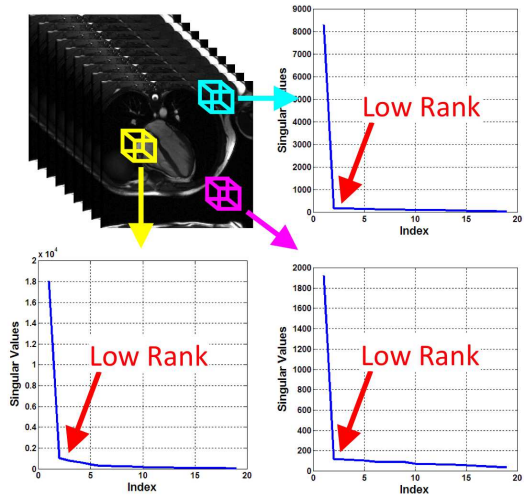}
\caption{
Singular values of the Casorati matrices formed from three different $8\times 8\times 19$ block sets extracted from a cine cardiac sequence, including: a static region (cyan); a dynamic region (yellow); and background noise (magenta).
}
\label{figure:local_svt}
\end{figure}

Let $\vc{R}_{b}$ be a binary operator that extracts
$k^{2}$ rows from a matrix corresponding to a $k\times{k}$ spatial
block, specified by index $b$, within each image. Block-wise SVT can
be defined as
\begin{equation}
\label{eq:llr_prox}
\bsvt_{\lambda}(\vc{Y}) = c^{-1}\sum_{b\in\Omega} \vc{R}_{b}^{*}\, \svt_{\lambda}(\vc{R}_{b}\vc{Y}),
\end{equation}
where $\Omega$ denotes a set of (potentially overlapping) blocks that uniformly tiles the image
domain, i.e., $\sum_{b\in\Omega} \vc{R}_{b}^{*} \vc{R}_{b} =
c\id_{n^2\times{n^2}}$, $c>0$. In words,~\eqref{eq:llr_prox} performs
SVT on a family of submatrices of $\vc{Y}$ and accumulates a weighted
sum of the results. Of course, for $k=n$ and
$|\Omega|=1$,~\eqref{eq:llr_prox} resorts to standard SVT. The
unbiased risk estimator developed in the previous section readily
extends for this generalized SVT model. By linearity,
\begin{equation}
\label{eq:div_G}
\div\, \bsvt_{\lambda}(\vc{Y}) = c^{-1}\sum_{b\in\Omega} \div\, \vc{R}_{b}^{*}\, \svt_{\lambda}(\vc{R}_{b}\vc{Y}).
\end{equation}
Extending the identity~(10) from~\cite{Blu2007} for matrices then asserts
\begin{equation}
\label{eq:div_G_a}
\div\, \vc{R}_{b}^{*}\, \svt_{\lambda}(\vc{R}_{b}\vc{Y})
= \sum_{ij} \frac{\partial (\vc{R}_{b}^{*}\, \svt_{\lambda}(\vc{R}_{b}\vc{Y}))_{ij}}{\partial Y_{ij}}
= \sum_{ij} \frac{\partial\, \svt_{\lambda}(\vc{R}_{b}\vc{Y})_{ij}}{\partial (\vc{R}_{b}\vc{Y})_{ij}} (\vc{R}_{b}\vc{R}_{b}^{*})_{ii}.
\end{equation}
Now, observing that $\vc{R}_{b}\vc{R}_{b}^{*}=\id _{k^2\times{k^2}}$,
$\forall{b}\,\in\Omega$, we see that $\div\, \vc{R}_{b}^{*}\,
\svt_{\lambda}(\vc{R}_{b}\vc{Y}) = \div\,
\svt_{\lambda}(\vc{R}_{b}\vc{Y})$, where the latter term is as in~\eqref{eq:weak_div_for_sure}
for $\vc{Y}\in\R^{n^2\times{t}}$ or~\eqref{eq:weak_div_for_sure_cpx} for
$\vc{Y}\in\C^{n^2\times{t}}$. Define the mean-squared error of
$\bsvt_{\lambda}$ as
\begin{equation}
\label{eq:risk_mse2}
\text{MSE}(\lambda) = \ev\|\vc{X}_0 - \bsvt_{\lambda}(\vc{Y})\|_F^2,
\end{equation}
and the singular value decomposition $\vc{R}_{b}\vc{Y} =
\vc{U}_{b}\vc{\Sigma}_{b}\vc{V}_{b}^{*}$. An unbiased estimator of~\eqref{eq:risk_mse2} is
\begin{equation}
\label{eq:sure_for_svt2}
\sure(\bsvt_\lambda)(\vc{Y}) = -\beta{mn}\tau^{2}
+c^{-2} \left\|
  \sum_{b\in\Omega} \vc{R}_{b}^{*}\vc{U}_{b}\mathcal{H}_{\lambda}(\vc{\Sigma}_{b})\vc{V}_{b}^{*}
  \right\|_F^2
 + 2\tau^{2}c^{-1} \sum_{b\in{\Omega}} \div\, \svt_\lambda(\vc{R}_{b}\vc{Y}),
\end{equation}
where $\mathcal{H}(\vc{\Sigma}_{b})_{ij} =
\min(\lambda,\vc{\Sigma}_{ij}) = \ind(i=j)\min(\lambda,\sigma_{i})$
and $\sigma_{i} = (\Sigma_{b})_{ii}$. For
$\vc{Y}\in\R^{n^2\times{t}}$, $\beta=1$; otherwise, for
$\vc{Y}\in\C^{n^2\times{t}}$, $\beta=2$.

We now present three MRI examples, one on simulated data and two on
real clinical data, and demonstrate the utility of the developed
unbiased risk estimators for automatic parameterization of SVT-based
denoising.

\subsubsection*{Example 1: PINCAT Numerical Phantom}

\begin{figure}[t!]
\centering
\includegraphics[width=0.95\textwidth]{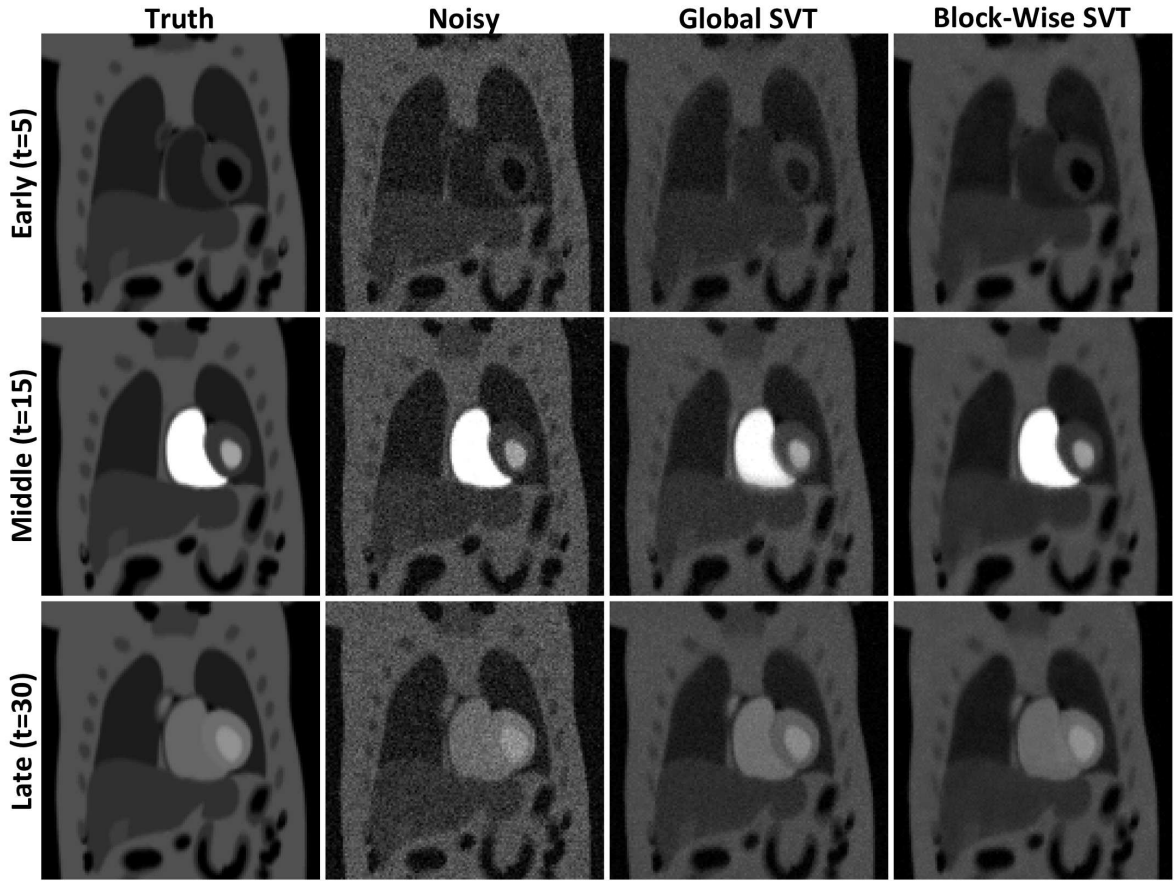}
\caption{Comparison of MSE-optimized global and block-wise SVT for denoising the complex-valued numerical PINCAT phantom. All images are identically windowed and leveled.
}
\label{figure:pincat_series}
\end{figure}

Initial evaluation of SURE for SVT was performed on the
physiologically-improved NCAT (PINCAT) numerical
phantom~\cite{Shariff2007}, which simulates a first-pass myocardial
perfusion real-time MRI series. In particularly, the free-breathing
model ($n=128$, $t=50$) available in the kt-SLR software
package~\cite{Lingala2011} was adapted to include a spatially-smooth
and temporally-varying phase field such that the target signal was
complex valued. Complex \iid\, Gaussian noise ($\tau$ = 30) was then
added to the image data. Both standard, or global, and block-wise SVT
were each executed at 101 values equispaced over
$[10^{-1},10^{7}]$. Block-wise SVT was performed with $k=7$ and
$\Omega$ comprising one block for each image pixel, under periodic
boundary conditions, such that $|\Omega|=n^{2}$.

\begin{figure}[t!]
\centering
\includegraphics[width=0.95\textwidth]{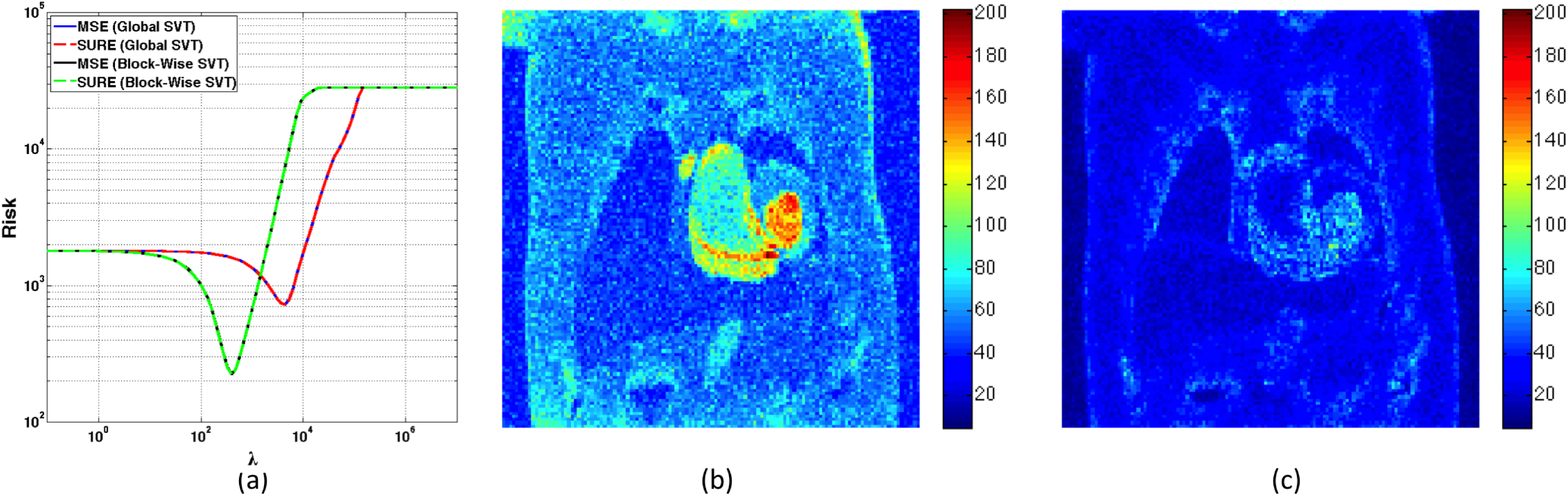}
\caption{(a) Plots of MSE and SURE for global and block-wise SVT as a
  function of threshold value, $\lambda$; the minimizing parameters
  were used to generate the images in Fig.~\ref{figure:pincat_series}.
  (b) and (c) show the worst-case error through time of the
  MSE/SURE-optimized global and block-wise SVT results, respectively.
}
\label{figure:pincat_plots}
\end{figure}

Figure~\ref{figure:pincat_series} shows early, middle, and late time
frames from the PINCAT series for the noise-free (truth), noisy, and
SVT-denoising results. The threshold values used to generate the SVT
results were selected as the MSE/SURE-minimizers in Figure~\ref{figure:pincat_plots}a. Also observe in Figure~\ref{figure:pincat_plots}a that SURE provides a precise estimate of
MSE for both the global and block-wise SVT models. The high accuracy
exhibited in this case can be attributed to the high dimensionality of
the MRI series data. Note that both global and block-wise SVT yield
strong noise reduction generally preserve both morphology and
contrast. However, block-wise SVT simultaneously demonstrates a
greater degree of noise removal and fidelity to ground truth. In
particular, note the relative contrast of the various components of
the heart in late frame results. The first observation is
corroborated by Figure~\ref{figure:pincat_plots}a, which shows that
block-wise SVT is able to achieve a lower MSE than global SVT. The
second observation is corroborated by Figures~\ref{figure:pincat_plots}b-c, which show the worst-case absolute error
(compared to ground truth) through time for the two SVT setups. Clearly, global SVT exhibits higher residual error than block-wise
SVT, particularly in areas of high motion near the myocardium. The
difference between these results can be attributed to the matrix
anisotropy problem discussed earlier in this section. Thus, SURE can
also be used to automatically-determine the block-size setting as well
as the threshold value.

\subsubsection*{Example 2: Cine Cardiac Imaging}

\begin{figure}[t!]
\centering
\includegraphics[width=0.95\textwidth]{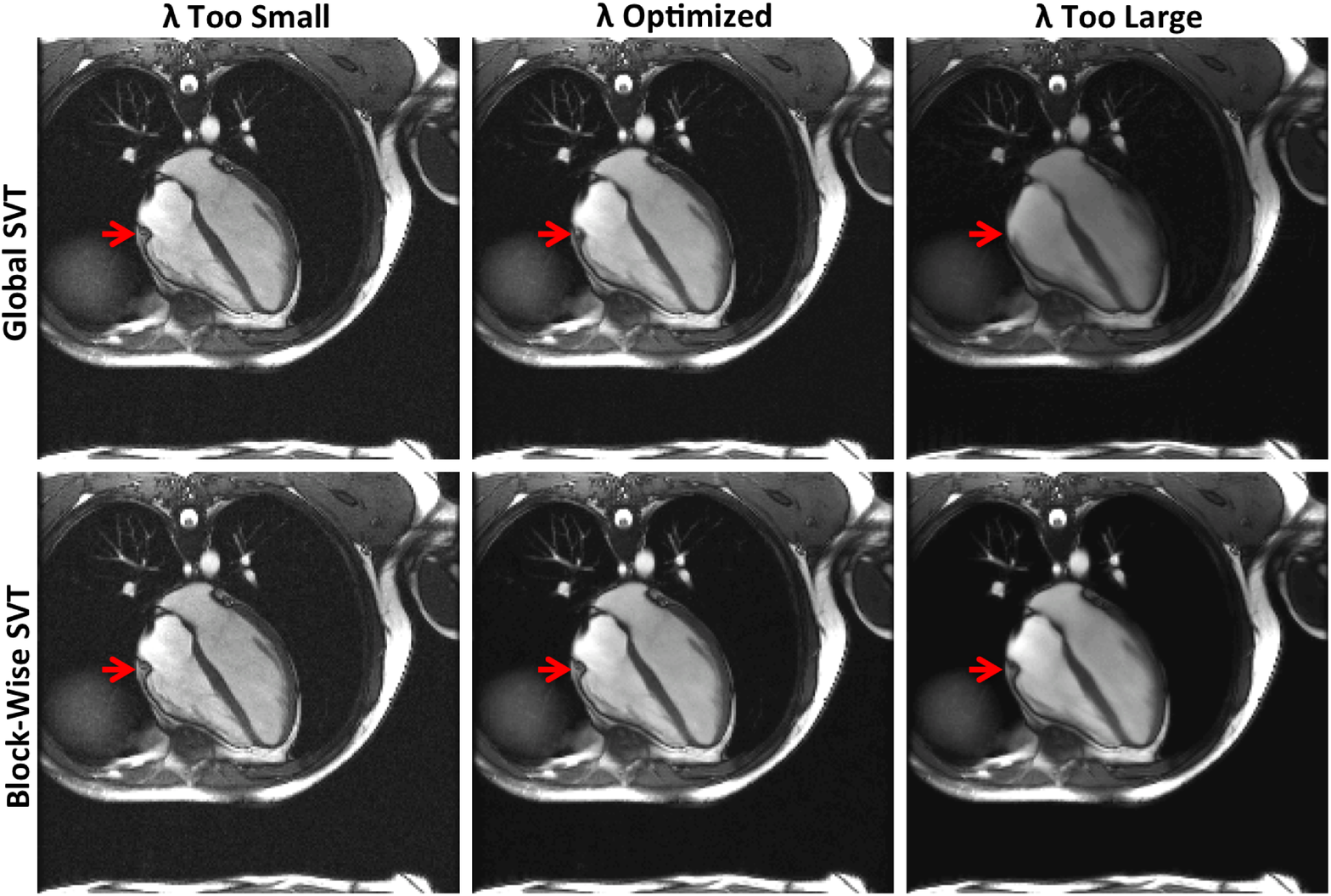}
\caption{Global and block-wise SVT denoising results for a bSSFP
  long-axis cine cardiac MRI series. The effect of selecting too high
  or low a threshold value, as compared to the MSE/SURE optimal value
  is demonstrated. The red arrow identifies the tricuspid valve of
  the heart, which exhibit a large degree of motion through the
  sequence. All images are identically windowed and leveled.}
\label{figure:cine_series}
\end{figure}

In the second experiment, a bSSFP long-axis cine sequence ($n=192$,
$t=19$) acquired at 1.5 T using an phased-array cardiac receiver coil
($\l=8$ channels) was denoised via SVT. The Casorati matrices for
individual data channels, which have undergone pre-whitening, are stacked to form a single $ln^2\times{t}$
matrix. Assuming that the spatial sensitivity of the receiver
channels does not vary substantially through time, the rank of this
composite matrix will be equal to that for any individual channel. For visualization purposes, multi-channel
denoising results were compressed by calculating the
root-sum-of-squares image across the channel dimension. Background
noise was determined using a manually-drawn region-of-interest (ROI)
to have $\tau$ = 0.67.  In this example, block-wise SVT
was performed with $k=5$ and used the same block set as in Example 1.
The parameter sweep was executed for 101 values equispaced over $[10^{-3},10^{5}]$.
In this example, the ground truth is unknown, and thus the MSE cannot be computed
to be compared against SURE. However, inspection of Figure~\ref{figure:cine_plots}a reveals that the same qualitative behavior
seen for SURE in the numerical phantom example is observed here as
well.

\begin{figure}[t!]
\centering
\includegraphics[width=0.95\textwidth]{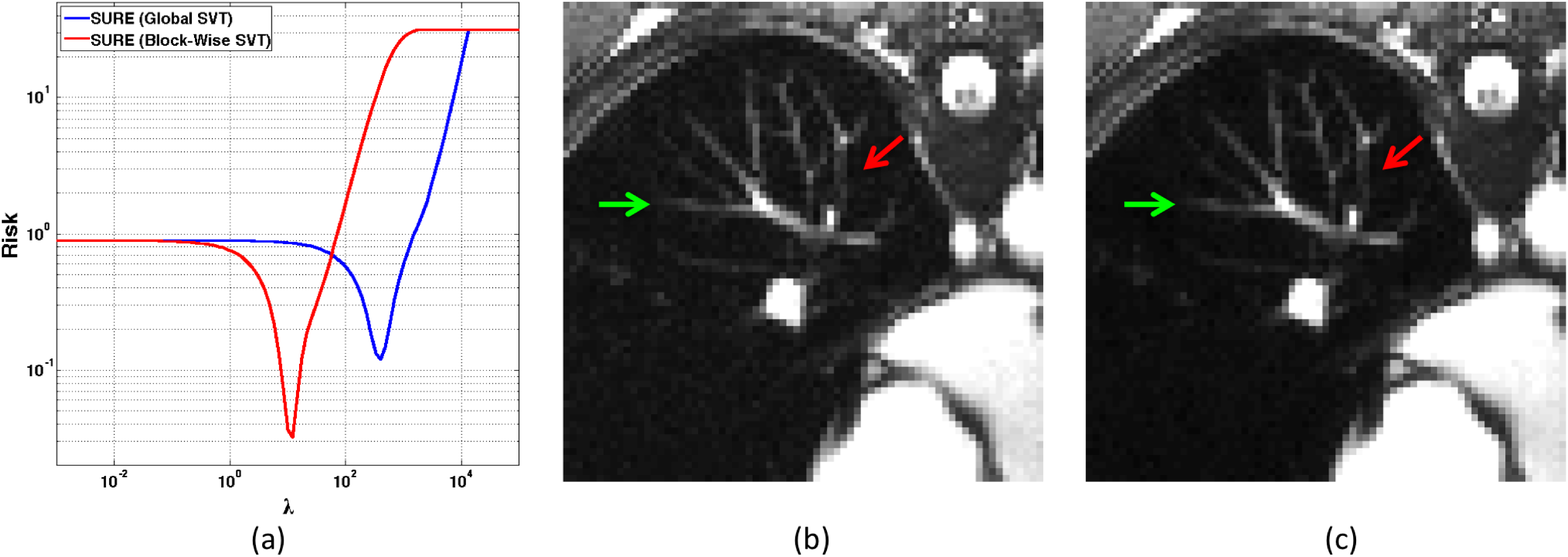}
\caption{(a) SURE for global and block-wise SVT as a function of
  threshold value, $\lambda$; the minimizing parameters were used to
  generate the images in Fig.~\ref{figure:cine_series}. (b) and (c)
  show enlargements of the pulmonary branching vessels for the
  SURE-optimized global and block-wise SVT results in
  Fig.~\ref{figure:cine_series}, respectively. Red and green arrows
  highlight improved vessel conspicuity. Both images are windowed and
  leveled identically.}
\label{figure:cine_plots}
\end{figure}

Figure \ref{figure:cine_series} demonstrates the effect of
sub-optimally selecting the threshold value for both global and
block-wise SVT.  In particular, over- and under-estimates, along with
the SURE-optimal values are applied within SVT.  For both SVT
strategies, under-estimation of $\lambda$ fails to remove noise as
expected.  Conversely, over-estimation of $\lambda$ leads to
spatio-temporal blurring by both strategies albeit in different
manners.  In global SVT, temporal blurring occurs near areas of high
motion like the tricuspid valve of the heart (indicated via the red
arrow).  However, less active areas like the pulmonary branching
vessels, seen just above the heart, are undistorted.  Some background
noise also remains visible.  In the block-wise SVT result, these
characteristics are essentially reversed---minimal temporal blurring
is induced but there is noticeable spatial blurring and loss of
low-contrast static features.  Noise is, however, strongly reduced.  A
compromise between these extremes is made by the SURE-optimized
results.  As suggested by Figure \ref{figure:cine_plots}a, block-wise
SVT offers stronger noise suppression (see Figures
\ref{figure:cine_plots}b-c) without inducing spatial or temporal blur,
the latter which is seen even in the optimized global SVT result.
This example highlights the sensitivity of SVT-denoising performance
on parameter selection, and the importance of having a disciplined
framework for choosing threshold values.  Also of note is that, even
after empirical pre-whitening, multi-channel MRI noise may not be {\em
  exactly} iid Gaussian.  Nonetheless, SURE allows production of
extremely reliable results.

\subsubsection*{Example 3: First-Pass Myocardial Perfusion}

\begin{figure}[t!]
\centering
\includegraphics[width=0.95\textwidth]{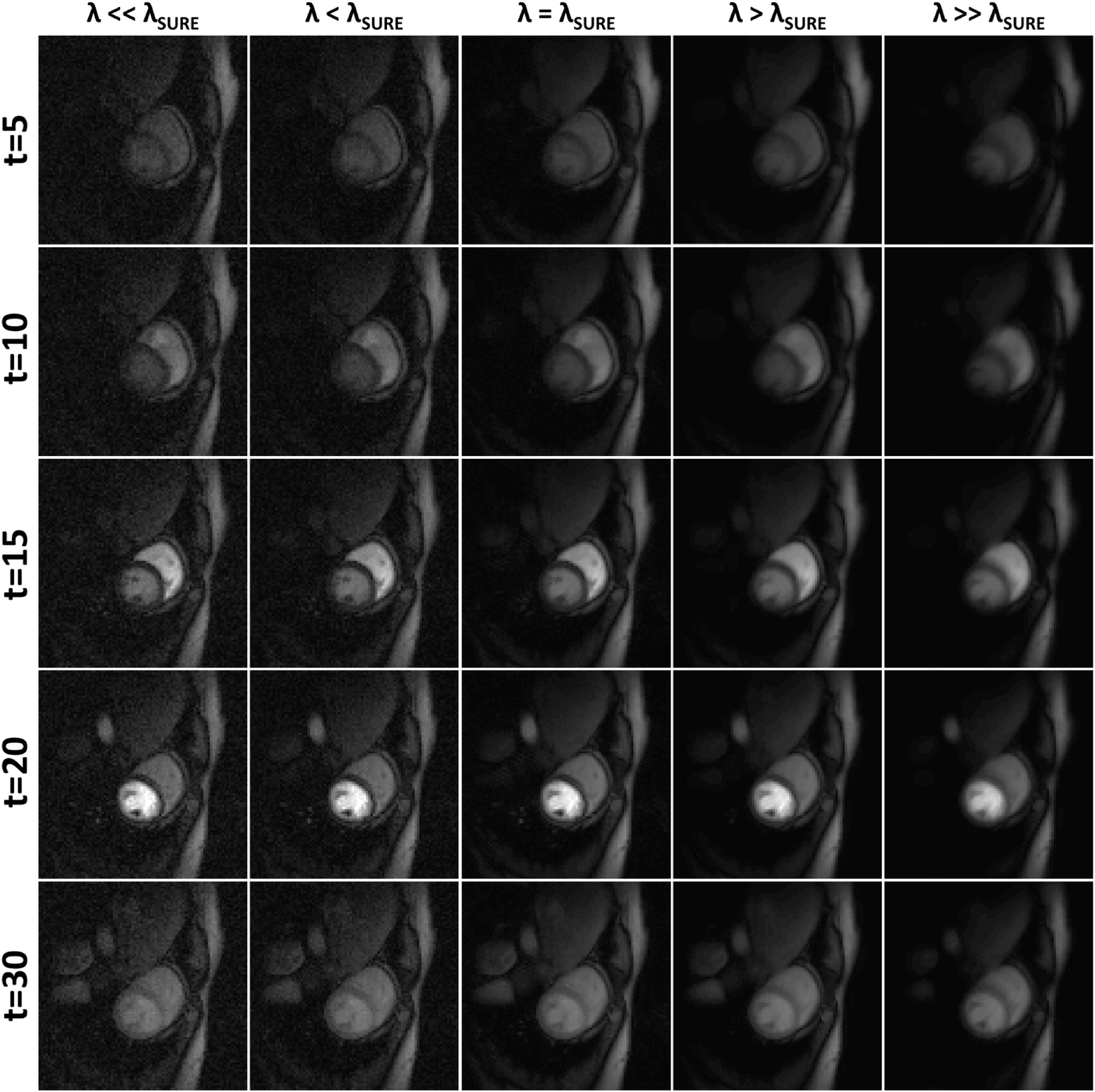}
\caption{Block-wise SVT denoising results for GRE short-axis first-pass myocardial
perfusion sequence. The effect of selecting too high
 or low a threshold value, as compared to the MSE/SURE optimal value,
 is demonstrated. All images are identically windowed and leveled.
}
\label{figure:FPMP_series}
\end{figure}

The third denoising experiment was performed on the single-channel, first-pass
myocardial perfusion real-time MRI sequence ($n_{x}=190$, $n_{y}=90$, $t=70$)
that is also provided with the kt-SLR software package~\cite{Lingala2011}. To simulate
a low-field acquisition (1.0 T), such as with an open or interventional MRI system, complex
Gaussian noise was added to the data originally acquired at 3.0 T using a GRE sequence.
Following this addition, the noise level was estimated at $\tau = 0.105$. Since the
duration of this exam exceeded feasible breath-hold time, there is substantial
respiratory motion at the end of the series. For this example, only block-wise SVT ($k=6$)
denoising was performed, and executed at 101 threshold values equispaced over $[10^{-3},10^{5}]$.
As in Example 2, $\Omega$ comprises one block for each image pixel, and only SURE can be
computed due to the absence of ground truth.

\begin{figure}[t!]
\centering
\includegraphics[width=0.45\textwidth]{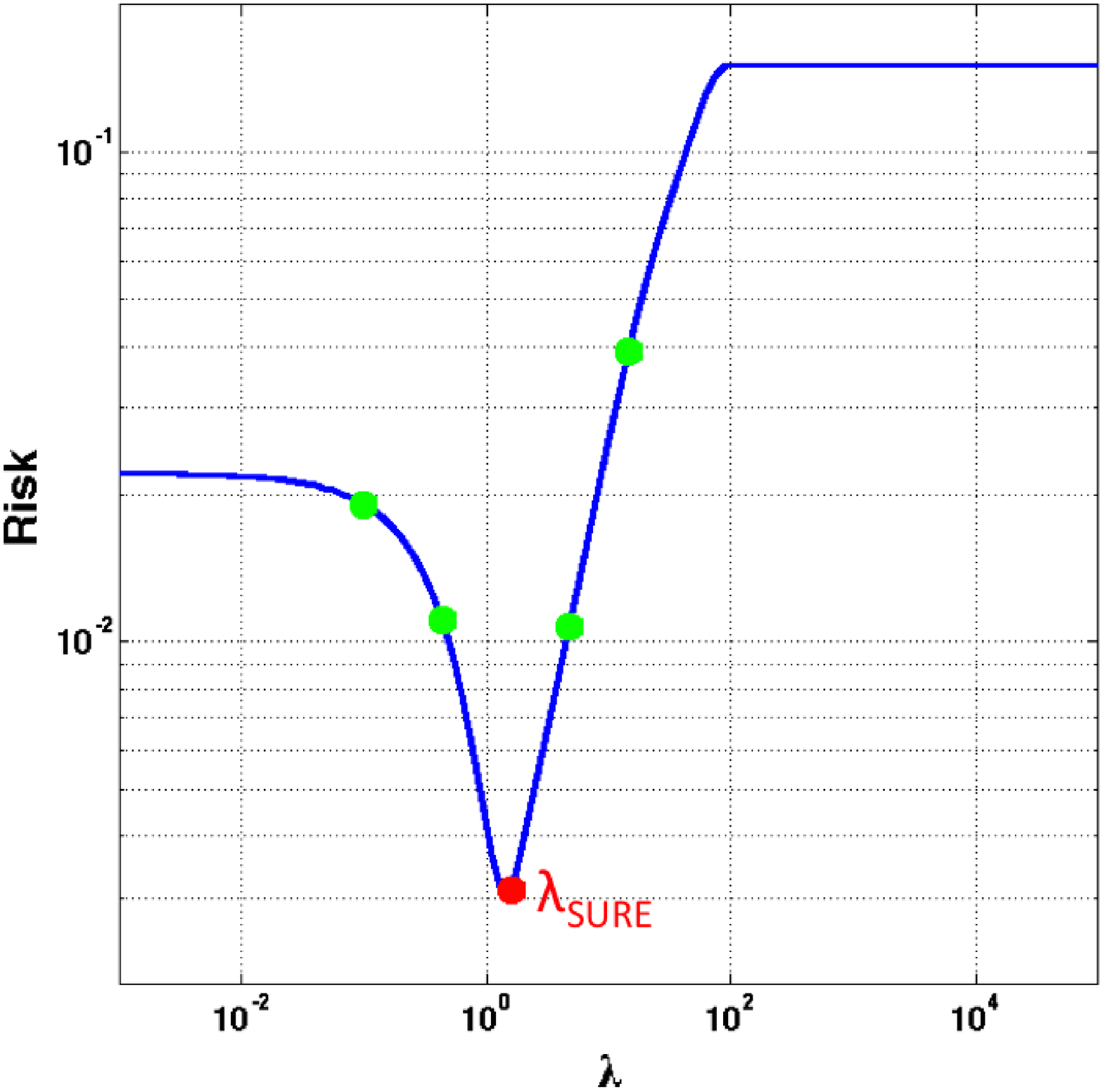}
\caption{ SURE for block-wise SVT as a function of threshold value,
  $\lambda$; the red dot indicates the optimal threshold value used in
  Fig.~\ref{figure:FPMP_series}, and green dots correspond to the
  employed under- and over-estimated values.  }
\label{figure:FPMP_plots}
\end{figure}

Mirroring Figure~\ref{figure:cine_series},
Figure~\ref{figure:FPMP_series} demonstrates the effect of threshold
selection on different image frames from the perfusion series. In
particular, images showing the transition of contrast from the right
to left vertical, as well as later-stage onset of myocardial blush are
shown. Of the 101 tested threshold values, 5 threshold settings
corresponding to 2 over-estimation, 2 under-estimation, and the
SURE-optimal value are depicted (see
Figure~\ref{figure:FPMP_plots}). Under all display settings, temporal
contrast dynamics are well preserved; however, threshold setting has a
marked effect on noise level, spatial resolution, and contrast. As
before, employing SVT with an under-estimated threshold fails to
remove noise, as evident in the first two columns of
Figure~\ref{figure:FPMP_series}.  Conversely, employing SVT with an
over-estimated threshold will remove noise but also induce both
spatial blurring and contrast loss. This is particularly evident in
the 4th image row ($t=20$).  At high threshold values, the contrast of
pulmonary vasculature is diminished, and there is visible blurring of
the papillary muscles, which appear as dark spots in the
contrast-enhanced left ventricle.  Finally, the SVT result obtained
with the SURE-optimal threshold represents an ideal balance between
these extremes, offering strong noise reduction without degradation of
important anatomical features.  As suggested by
Figure~\ref{figure:FPMP_plots}, there may only exist a narrow
parameter window in which effective denoising can be achieved, and the
presented unbiased risk estimators can greatly aid in the
identification of these optimal settings.

\subsection{Extensions and generalizations}

The three examples in the previous subsection demonstrate the utility
of the unbiased risk estimation for SVT-based denoising of cardiac MRI
series data. Of course, this methodology can also be applied to any
other dynamic or parametric MRI series, including those discussed
earlier such as functional MRI. Although exhaustive search over a
wide range of threshold values was performed for the sake of
exposition, the observed unimodality of the risk functional suggests
that practical automated parameter selection for SVT denoising could
be accelerated using a bisection strategy such as golden section
search.

In addition to optimizing over a range of threshold values, unbiased
risk estimation can also be used to identify the optimal block-size
when using the local SVT model in~\eqref{eq:llr_prox}. For example,
one could imagine performing a sweep of $\lambda$ over a prescribed
range of values for a collection of block-size values, and computing
the lower envelope of this set of risk measures to investigate best
achievable performance as a function of block-size. In addition to
simply optimizing a denoising setup, this type of information can be
used to guide the development of rank-based reconstruction frameworks
for undersampled dynamic and/or parallel MRI.

\section{SURE Formulas for Singular Value Thresholding}
\label{section:sure_for_svt}

In \cite{Stein81}, Stein gave a formula for an unbiased estimate of
the mean-squared error of an estimator obeying a weak
differentiability assumption and mild integrability conditions. Before
reviewing this, recall that a function $g:\R^p \mapsto \R$ is said to
be weakly differentiable with respect to the variable $x_i$ if there
exists $h:\R^p \mapsto \R$ such that for all compactly supported and
infinitely differentiable functions $\varphi$,
\begin{equation*}
    \int \varphi(x) h(x)\, dx = -\int\frac{\partial\varphi(x)}{\partial x_{i}} g(x)\, dx. 
\end{equation*}
Roughly speaking, the derivatives can fail to exist over regions of
Lebesgue measure zero.
\begin{proposition}[\cite{Stein81,JohnstoneGE}]
\label{proposition:sure}
Suppose that $Y_{ij}\stackrel{\iid}{\sim} \normaldist(X_{ij},
1)$. Consider an estimator $\hat{\vc{X}}$ of the form $\hat{\vc{X}} =
\vc{Y} + g(\vc{Y})$, where $g_{ij}:\R^{m\times n}\mapsto\R$ is weakly
differentiable with respect to $Y_{ij}$ and
    \begin{equation*}
      \ev\left\{\left|Y_{ij}g_{ij}(\vc{Y})\right| + \left|\frac{\partial}{\partial Y_{ij}}g_{ij}(\vc{Y})\right|\right\} < \infty,
    \end{equation*}
    for $(i,j)\in \idx := \{1, \ldots, m\} \times \{1, \ldots
    n\}$. Then
    \begin{equation}\label{eq:proposition:sure}
      \ev\|\hat{\vc{X}} - \vc{X}\|_F^2 = \ev\bigl\{mn + 2\div\, (g(\vc{Y})) + \|g(\vc{Y})\|_F^2\bigr\}.
    \end{equation}
\end{proposition}

The rest of this section establishes that $\svt$ obeys these
assumptions whereas the deduction of a closed-form expression for its
divergence is deferred to Section~\ref{section:general_theory}. To
achieve this, we record a simple fact that will allow us to study both
the real- and complex-valued cases within the same framework. We
identify below $\R^{m\times n}$ with $\R^{mn}$.
\begin{lemma}\label{lemma:lipschitz_implies_weak_differentiability}
  Let $g :\R^{m\times n}\mapsto\R^{m\times n}$ with components
  $g_{ij}:\R^{m\times n}\mapsto \R$.  Assume $g(\vc{0}) = \vc{0}$. If
  $g$ is Lipschitz with constant $L > 0$, i.e.,
    \begin{equation*}
      \forall\, \vc{X},\vc{Y}:\quad \|g(\vc{X}) - g(\vc{Y})\| \leq L
\|\vc{X} - \vc{Y}\|,
    \end{equation*}
    for some norm in $\R^{m\times n}$, then each $g_{ij}$ is weakly
    differentiable Lebesgue-a.e.~in $\R^{m\times n}$.  Furthermore, if
    $\vc{Y}\in\R^{m\times n}$ is distributed as in
    \eqref{eq:mvn_mean_estimation}, then for all pairs $(i,j)$,
    \begin{equation*}
      \ev\left\{\left|Y_{ij}g_{ij}(\vc{Y})\right|\right\} < \infty \quad \text{ and } \quad
      \ev\left\{\left|\frac{\partial}{\partial Y_{ij}}g_{ij}(\vc{Y})\right|\right\} < \infty.
    \end{equation*}
\end{lemma}
\begin{proof}
  Since all norms are equivalent in finite dimensions, $g$ is also
  Lipschitz with respect to the Frobenius norm. Hence,
    \begin{equation}\label{eq:g_lipschitz_frobenius}
      \forall\, \vc{X}, \vc{Y}:\quad \|g(\vc{X}) - g(\vc{Y})\|_F \leq L
      \|\vc{X} - \vc{Y}\|_F \quad \Longrightarrow \quad |g_{ij}(\vc{X}) - g_{ij}(\vc{Y})| \leq L
      \|\vc{X} - \vc{Y}\|_F,
    \end{equation}
    for some constant $L$.
By Rademacher's theorem (see Section 3.1.2 in~\cite{EvansGariepyMT})
each component is differentiable Lebesgue-a.e.~in $\R^{m\times
  n}$. Furthermore, these components are weakly differentiable (see
Theorem 1 and 2, Section 6.2, in~\cite{EvansGariepyMT}), and both the
derivatives and weak derivatives are Lebesgue-a.e.~equal.

Regarding the integrability, since $g(\vc{0}) = \vc{0}$,
\eqref{eq:g_lipschitz_frobenius} yields $\|g(\vc{Y})\|_F^2 \leq
L^2\|\vc{Y}\|_F^2$ and we deduce
    \begin{equation*}
      \ev\|g(\vc{Y})\|_F^2 \leq L^2\ev\|\vc{Y}\|_F^2 = L^2(mn + \|\vc{X}_{0}\|_F^2) < \infty.
    \end{equation*}
    Furthermore, the Cauchy-Schwarz inequality gives
    \begin{equation*}
      \ev\{|Y_{ij} g_{ij}(\vc{Y})|\} \leq \ev\{Y_{ij}^2\}^{1/2}\ev\{g_{ij}(\vc{Y})^2\}^{1/2} \leq L\left[(1 + X_{0, ij}^2)(mn + \|\vc{X}_{0}\|_F^2)\right]^{1/2} < \infty.
    \end{equation*}
    Finally, \eqref{eq:g_lipschitz_frobenius} asserts that the
    derivatives of $g_{ij}$---whenever they exist---are bounded by
    $L$. Hence,
    \begin{equation*}
    	\ev\left\{\left|\frac{\partial}{\partial Y_{ij}}g_{ij}(\vc{Y})\right|\right\} \leq \ev\left\{\left|\frac{\partial}{\partial Y_{ij}}g_{ij}(\vc{Y})\right|^2\right\}^{1/2}  \leq L,
    \end{equation*}
    which concludes the proof.
\end{proof}

For pedagogical reasons, we work with singular value thresholding
first, before discussing more general spectral estimators in Section~\ref{section:general_theory}.  Writing $g(\vc{Y}) =
\svt_\lambda(\vc{Y}) - \vc{Y}$ then, the conditions become
\begin{equation*}
  \ev\left\{\left|Y_{ij}g_{ij}(\vc{Y})\right| + \left|\frac{\partial g_{ij}(\vc{Y})}{\partial Y_{ij}}\right|\right\}\leq\ev\left\{\left|Y_{ij}\svt_{\lambda}(\vc{Y})_{ij}\right| +
    \left|\frac{\partial}{\partial Y_{ij}}\svt_{\lambda}(\vc{Y})_{ij}\right|\right\} + \ev\{Y_{ij}^2\} + 1,
\end{equation*}
and we only need to study the weak differentiability and integrability
of $\svt$.

\subsection{The real case}
\label{section:sure_for_svt:real}

\begin{lemma}\label{lemma:svt_meets_sure_conditions}
  The mapping $\svt_\lambda$, $\lambda > 0$, obeys the
  assumptions of Proposition~\ref{proposition:sure}.
\end{lemma}
\begin{proof}
  By definition, $\svt_\lambda$ is the proximity function of the
  nuclear norm. Since the nuclear norm is a convex, proper and lower
  semi-continuous function, $\svt_\lambda$ is Lipschitz and
  non-expansive (see Chapter 6, Section 30
  in~\cite{RockafellarCA}). Applying
  Lemma~\ref{lemma:lipschitz_implies_weak_differentiability} proves
  the claim.
\end{proof}


The mapping $\svt_\lambda$ may not be differentiable
everywhere. Fortunately, since we are interested in weak
differentiability, we can discard sets of Lebesgue measure zero. For
instance, model~\eqref{eq:mvn_mean_estimation} guarantees that
$\vc{Y}$ is simple and has full-rank with probability one.
Further, since $\svt_\lambda$ acts by applying soft-thresholding
to each singular value, it seems natural that matrices with at least
one singular value equal to $\lambda$ will be a cause of
concern. Consider then
\begin{equation*}
  \F := \{\vc{X}\in\R^{m\times n}:\, \text{$\vc{X}$ is simple, has full rank and no singular value exactly equal to $\lambda$}\}.
\end{equation*}
It is clear that the complement of $\F$ has Lebesgue measure
zero. Consequently, we can restrict ourselves to the open set $\F$ to
look for a divergence. The result below---a corollary of the general
Theorem~\ref{theorem:div_simple}---establishes that $\svt$ is
differentiable over $\F$, and thus we have a divergence in the usual
sense.
\begin{corollary}\label{theorem:svt_div}
  The mapping $\svt_\lambda$ is differentiable over $\F$ and we have
    \begin{equation}\label{eq:lemma:svt_div:svt_div}
      \div\, (\svt_\lambda(\vc{X})) = \sum_{i = 1}^{\min(m, n)}\left[\ind\{\sigma_i > \lambda\} + |m-n|\left(1 - \frac{\lambda}{\sigma_i}\right)_+\right]+ 2\sum_{i\neq j,\, i, j = 1}^{\min(m,
      n)} \frac{\sigma_i (\sigma_i-\lambda)_+}{\sigma_i^2 - \sigma_j^2}.
    \end{equation}
\end{corollary}
Since $\svt_\lambda$ is of the form~\eqref{eq:spectral_function} with
    \begin{equation}
\label{eq:fi}
f_i(\sigma)     =   (\sigma - \lambda)_+, \quad f_i'(\sigma) = \ind\{\sigma > \lambda\},
\end{equation}
for $i = 1,\ldots, \min(m, n)$ and $\sigma\neq\lambda$, we see that it
is differentiable at the singular value matrix of any element of
$\F$. Because the elements of $\F$ are also simple and full-rank, we
conclude $\svt_\lambda$ is differentiable over $\F$ by
Lemma~\ref{lemma:diff_svd}. Therefore, the proof
of~\eqref{eq:lemma:svt_div:svt_div} is a special case of the general
Theorem~\ref{theorem:div_simple} from
Section~\ref{section:general_theory}.

\subsection{The complex case}
\label{section:sure_for_svt:cpx}

In the complex case, we need to analyze the conditions for SURE with
respect to model~\eqref{eq:mvn_mean_estimation_cpx}, which can be
expressed in vector form as
\begin{equation}\label{eq:mvn_mean_estimation_vector_cpx}
  \begin{bmatrix}\real(\vc{Y}) \\ \imag(\vc{Y})\end{bmatrix} = \begin{bmatrix} \real(\vc{X}_0) \\ \imag(\vc{X}_0)\end{bmatrix} + \begin{bmatrix}\real(\vc{W})\\ \imag(\vc{W})\end{bmatrix}.
   \end{equation}
   We thus need to study the existence of the weak partial derivatives
\begin{equation*}
    \frac{\partial}{\partial\real(Y_{ij})}\real(\svt_{\lambda})_{ij}\qquad\text{and}\qquad \frac{\partial}{\partial\imag(Y_{ij})}\imag(\svt_{\lambda})_{ij},
\end{equation*}
and whether they satisfy our integrability conditions, as these are the only
ones that play a role in the divergence.

\begin{lemma}\label{lemma:svt_meets_sure_conditions_cpx}
  The mappings $\svt_\lambda$, $\real{(\svt_\lambda)}$ and
  $\imag{(\svt_\lambda)}$ obey the assumptions of
  Proposition~\ref{proposition:sure}.
\end{lemma}
\begin{proof}
  As in the real case, $\svt_\lambda$ is the prox-function of
  the nuclear norm and is thus Lipschitz and non-expansive (see
  Proposition 2.27 in~\cite{BaushkeCombettesCMO}). Hence, the real and
  imaginary parts are Lipschitz. Applying
  Lemma~\ref{lemma:lipschitz_implies_weak_differentiability} to both
  the real and imaginary components proves the claim.
\end{proof}

For the divergence, we write
\begin{equation}\label{eq:def_div_cpx}
  \div\, (\svt_\lambda(\vc{X})) = \div_{\real(\vc{X})}\,(\real(\svt_\lambda)(\vc{X})) + \div_{\imag(\vc{X})}\,(\imag(\svt_\lambda)(\vc{X})).
\end{equation}
As in the real case, we can discard sets of Lebesgue measure
zero. Consider then the analogue of $\F$, whose complement has
Lebesgue measure zero.
\begin{corollary}\label{theorem:svt_div_cpx}
  The mapping $\svt_\lambda$ is differentiable over $\F$, and we have
    \begin{equation}\label{eq:lemma:svt_div_cpx:svt_div}
      \div\, (\svt_\lambda(\vc{X})) = \sum_{i = 1}^{\min(m, n)}\left[\ind\{\sigma_i > \lambda\} + (2|m-n| + 1)\left(1 - \frac{\lambda}{\sigma_i}\right)_+\right]+ 4\sum_{i\neq j,\, i, j =
      1}^{\min(m, n)} \frac{\sigma_i (\sigma_i - \lambda)_+}{\sigma_i^2 - \sigma_j^2}.
    \end{equation}
\end{corollary}
The argument regarding the differentiability of $\svt_\lambda$ applies here
as well. The proof of the closed-form expression for the divergence is a
special case of the general Theorem~\ref{theorem:div_simple_cpx} from Section~\ref{section:general_theory} with $f_i$ as in \eqref{eq:fi}.

\section{Differentiability of Spectral Functions: General SURE
  Formulas}
\label{section:general_theory}

This section introduces some results concerning the differentiability
of spectral functions.\footnote{Here, we set aside any statistical
  interpretation to focus on differentiability and closed-form
  expression for the divergence of spectral functions.}
To simplify the exposition, we work with $m \times n$ matrices
$\vc{X}$ with $m \geq n$, as all our arguments apply with minor
modifications to $\vc{X}^*$.  Below, $\{\vc{E}^{ij}\}_{(i,j)\in\idx}$
is the canonical basis of $\R^{n\times m}$, $\id_{n\times n}$ is the
$n$-dimensional identity matrix and we set
\begin{equation*}
    \id_{m\times n} = \begin{bmatrix} \id_{n\times n} \\ \vc{0} \end{bmatrix}.
\end{equation*}
In the real and complex cases, we make use of the reduced and full
SVD. The reduced SVD is written as $\vc{X} =
\vc{U}\vc{\Sigma}\vc{V}^*$, where $\vc{U}$ is $m \times n$ with
orthonormal columns and $\vc{V}$ is unitary. In the full SVD, $\vc{X}
= \tilde{\vc{U}} \vc{\Sigma}\vc{V}^*$ where $\tilde{\vc{U}}$ is
$m\times m$ and extends $\vc{U}$ to a unitary matrix. Finally, we
recall from Section~\ref{section:introduction} that a spectral
function is of the form $f(\vc{X}) = \vc{U} f(\vc{\Sigma}) \vc{V}^*$
with
\begin{equation}\label{eq:spectral_function_on_sigma}
  f(\vc{\Sigma}) = \diag( f_1(\sigma_1), \ldots, f_n(\sigma_n) ),
\end{equation}
where we assume that each $f_i:\Rp\mapsto\Rp$ is differentiable. We
note that we can relax the differentiability assumption and only
require differentiability in a neighborhood of the spectrum of the
matrix under study.

To find a closed-form for the divergence of $f$, we first need to
determine whether $f$ is differentiable. Using notation from
differential calculus (see \cite{Edelman05}), recall that $f$ is
differentiable at $\vc{X}$ if there exists a linear mapping $\diff
f_{\vc{X}}$, called the differential of $f$ at $\vc{X}$, such that for
all $\vc{\Delta}$,
\begin{equation*}
  \lim_{\vc{\Delta} \rightarrow \vc{0}} \frac{\|f(\vc{X} + \vc{\Delta}) - f(\vc{X}) - \diff f_{\vc{X}}[\vc{\Delta}]\|_F}{\|\vc{\Delta}\|_F} = 0.
\end{equation*}
Differentiation obeys the standard rules of calculus and we have
\begin{equation}\label{eq:diff_product_g_rule}
  \diff f [\vc{\Delta}] = \diff \vc{U} [\vc{\Delta}]\, f(\vc{\Sigma})\, \vc{V}^* + \vc{U}\, \diff(f\circ\vc{\Sigma}) [\vc{\Delta}]\,\vc{V}^* + \vc{U}\, f(\vc{\Sigma})
  \, \diff\vc{V}[\vc{\Delta}]^*.
\end{equation}
where, to lighten the notation above, we have not indicated the point
$\vc{X}$ at which the differential is calculated (we shall use
subscripts whenever necessary).
By the chain rule, $\diff
(f\circ\vc{\Sigma})_\vc{X}[\vc{\Delta}] = \diff
f_\vc{\Sigma}[\diff\vc{\Sigma}_\vc{X}[\vc{\Delta}]]$ so
that the differentiability of $f$ depends upon the existence of
differentials for $\vc{U}$, $\vc{V}$ and $\vc{\Sigma}$ at $\vc{X}$
together with that of $f$ at $\vc{\Sigma}$.

Standard analysis arguments immediately establish that the function
$f$ is differentiable at $\vc{\Sigma}$ whenever $\vc{\Sigma}$ is
simple, as in this case there is no ambiguity as to which function
$f_i$ is being applied to the singular values. Furthermore, the
differentiability of the SVD of a simple matrix with full-rank is a
consequence of Theorem 1 and Theorem 2 in~\cite{Magnus85} applied to
$\vc{X}^*\vc{X}$ and $\vc{X}\vc{X}^*$~(see also~\cite{Lewis01} and~\cite{Sun02}). We summarize these arguments in
the following lemma.

\begin{lemma}\label{lemma:diff_svd}
  Let $\vc{X} = \vc{U}\vc{\Sigma}\vc{V}^*$ be a simple and full-rank
  matrix and $f$ be a spectral function. 
    \begin{enumerate}
    \item The factors of the SVD are differentiable in a neighborhood
      of $\vc{X}$.
    \item If $f$ is differentiable at $\vc{\Sigma}$, then it is
      differentiable at $\vc{X}$.
    \end{enumerate}
\end{lemma}

We now focus on determining the divergence of $f$ in closed-form. Since there
are differences between the real- and complex-valued cases, we treat them
separately.

\subsection{The real case}



Lemma~\ref{lemma:diff_svd} asserts that the SVD is differentiable at any
simple matrix $\vc{X}$ with full-rank. The following result shows that we can
determine the differentials of each factor in closed-form.

\begin{lemma}\label{lemma:diff_svd_simple}
  Let $\vc{X}$ be simple and full-rank. Then the differentials of
  $\vc{U}$ and $\vc{V}$ are given by
\[
\diff\vc{U}[\vc{\Delta}] = \tilde{\vc{U}}
\vc{\Omega}_{\tilde{\vc{U}}}[\vc{\Delta}], \qquad
\diff\vc{V}[\vc{\Delta}] = \vc{V} \vc{\Omega}_\vc{V}[\vc{\Delta}]^*,
\]
where $\vc{\Omega}_{\tilde{\vc{U}}}$ and $\vc{\Omega}_\vc{V}$ are
given in closed-form \eqref{eq:omega_entries} ($\diff\vc{U}$ is
independent of the choice of $\tilde{\vc{U}}$). Finally,
$\diff\vc{\Sigma}$ is given by \eqref{eq:diff_svalues}.
\end{lemma}
\begin{proof}
  We follow the same method as in~\cite{Edelman05}, which also appears
  in~\cite{Papadopoulo00}. Let $\vc{X}\in\R^{m\times n}$ ($m \geq
  n$). For any matrix $\vc{\Delta}\in \R^{m\times n}$, we have
    \begin{equation}\label{eq:diff_product_svd}
      \diff\vc{X}[\vc{\Delta}] = \diff\vc{U}[\vc{\Delta}] \vc{\Sigma} \vc{V}^* + \vc{U}\diff\vc{\Sigma}[\vc{\Delta}] \vc{V}^* + \vc{U}\vc{\Sigma} \diff\vc{V}[\vc{\Delta}]^*.
    \end{equation}
    Since $\tilde{\vc{U}}$ and $\vc{V}$ are orthogonal, it follows that
    \begin{equation*}
    	\tilde{\vc{U}}^*\diff\vc{X}[\vc{\Delta}]\vc{V} = \tilde{\vc{U}}^* \diff\vc{U}[\vc{\Delta}] \vc{\Sigma} + \tilde{\vc{U}}^* \vc{U} \diff\vc{\Sigma}[\vc{\Delta}] + \tilde{\vc{U}}^* \vc{U}
    \vc{\Sigma}
    \diff\vc{V}[\vc{\Delta}]^* \vc{V}.
    \end{equation*}
    Introduce the matrices
    \begin{equation*}
    	\vc{\Omega}_{\tilde{\vc{U}}}[\vc{\Delta}] = \tilde{\vc{U}}^* \diff\vc{U}[\vc{\Delta}]\qquad\text{and}\qquad\vc{\Omega}_\vc{V}[\vc{\Delta}] = \diff\vc{V}[\vc{\Delta}]^* \vc{V},
    \end{equation*}
    so that
    \begin{equation}\label{eq:diff_product_svd_omega}
      \tilde{\vc{U}}^* \vc{\Delta} \vc{V} = \vc{\Omega}_{\tilde{\vc{U}}}[\vc{\Delta}] \vc{\Sigma} + \id_{m\times n} \diff\vc{\Sigma}[\vc{\Delta}] + \id_{m\times n} \vc{\Sigma}
      \vc{\Omega}_\vc{V}[\vc{\Delta}],
    \end{equation}
    where we used the fact that $\diff\vc{X}$ is the identity, whence
    $\diff\vc{X} [\vc{\Delta}] = \vc{\Delta}$. Moreover,
    \begin{equation*}
      \vc{0} = \diff\id_{n\times n}(\vc{\Delta})  = \vc{V}^* \diff\vc{V}[\vc{\Delta}] + \diff\vc{V}[\vc{\Delta}]^* \vc{V}\quad\Longrightarrow\quad \vc{V}^* \diff\vc{V}[\vc{\Delta}] =
      -\diff\vc{V}[\vc{\Delta}]^*
      \vc{V},
    \end{equation*}
    which says that $\vc{\Omega}_\vc{V}[\vc{\Delta}]$ is
    anti-symmetric. Similarly,
    \begin{equation*}
    	\vc{0} = \diff\id_{n\times n}(\vc{\Delta})  = \vc{U}^* \diff\vc{U}[\vc{\Delta}] + \diff\vc{U}[\vc{\Delta}]^* \vc{U}\quad\Longrightarrow\quad \vc{U}^* \diff\vc{U}[\vc{\Delta}] =
    -\diff\vc{U}[\vc{\Delta}]^*
    \vc{U},
    \end{equation*}
    and thus the upper $n\times n$ block of
    $\vc{\Omega}_{\tilde{\vc{U}}}[\vc{\Delta}]$ is also
    anti-symmetric. This fact and the skew-symmetry of
    $\vc{\Omega}_\vc{V}[\vc{\Delta}]$ yield the relations
    \begin{eqnarray}
      \Omega_{\tilde{\vc{U}}, ii}[\vc{\Delta}] & = & 0,\qquad i = 1,\ldots, n,\nonumber\\
      \Omega_{\vc{V}, ii}[\vc{\Delta}] & = & 0,\qquad i = 1,\ldots, n,\nonumber\\
      (\tilde{\vc{U}}^*\vc{\Delta} \vc{V})_{ij} & = & \Omega_{\tilde{\vc{U}}, ij}[\vc{\Delta}]\sigma_{j} + \sigma_{i}\Omega_{\vc{V}, ij}[\vc{\Delta}],\qquad i\neq j,\quad i,j = 1,\ldots,n,
    \label{eq:omega_eqns_upper}\\
    	-(\tilde{\vc{U}}^*\vc{\Delta} \vc{V})_{ji} & = & \Omega_{\tilde{\vc{U}}, ij}[\vc{\Delta}]\sigma_{i} + \sigma_{j}\Omega_{\vc{V}, ij}[\vc{\Delta}],\qquad i\neq j,\quad i,j =
    1,\ldots,n,\label{eq:omega_eqns_lower}
    \end{eqnarray}
    and
    \begin{equation}\label{eq:diff_svalues}
    	\diff\sigma_i[\vc{\Delta}] = (\tilde{\vc{U}}^* \vc{\Delta} \vc{V})_{ii},\qquad i = 1,\ldots, n.
    \end{equation}
    In particular, \eqref{eq:omega_eqns_upper} and
    \eqref{eq:omega_eqns_lower} can be summarized in the system of
    linear equations
    \begin{equation*}
    	\begin{bmatrix}\sigma_{j} & \sigma_{i} \\ \sigma_{i} & \sigma_{j}\end{bmatrix}\begin{bmatrix}\Omega_{\tilde{\vc{U}}, ij}[\vc{\Delta}] \\ \Omega_{\vc{V}, ij}[\vc{\Delta}]\end{bmatrix} =
    \begin{bmatrix} (\tilde{\vc{U}}^*\vc{\Delta} \vc{V})_{ij} \\ -(\tilde{\vc{U}}^*\vc{\Delta} \vc{V})_{ji}\end{bmatrix},\qquad i\neq j,\quad i,j = 1,\ldots,n.
    \end{equation*}
    Since $\vc{X}$ is simple, the coefficient matrix is invertible,
    and thus
    \begin{equation}\label{eq:omega_entries}
    	\begin{bmatrix}\Omega_{\tilde{\vc{U}}, ij}[\vc{\Delta}] \\ \Omega_{\vc{V}, ij}[\vc{\Delta}]\end{bmatrix} = -\frac{1}{\sigma_{i}^2 - \sigma_{j}^2}\begin{bmatrix}\sigma_{j} & \sigma_{i} \\
    -\sigma_{i} & -\sigma_{j}\end{bmatrix}\begin{bmatrix} (\tilde{\vc{U}}^*\vc{\Delta} \vc{V})_{ij} \\ (\tilde{\vc{U}}^*\vc{\Delta} \vc{V})_{ji}\end{bmatrix},\qquad i\neq j,\quad i,j =
    1,\ldots,n.
    \end{equation}
    These relations determine both $\diff\vc{\Sigma}[\vc{\Delta}]$,
    $\vc{\Omega}_\vc{V}[\vc{\Delta}]$ and the upper $n\times n$
    block of $\vc{\Omega}_{\tilde{\vc{U}}}[\vc{\Delta}]$
    completely. The lower $(m-n)\times n$ block of
    $\vc{\Omega}_{\tilde{\vc{U}}}[\vc{\Delta}]$ is determined
    from~\eqref{eq:diff_product_svd_omega}, as
    \begin{equation}\label{eq:omega_eqns_nonsquare}
    	(\tilde{\vc{U}}^*\vc{\Delta} \vc{V})_{ij} = \Omega_{\tilde{\vc{U}}, ij}[\vc{\Delta}]\sigma_{j},\qquad i = n + 1,\ldots, m,\quad  j = 1,\ldots, n.
    \end{equation}
    Since $\vc{X}$ has full-rank, this completely determines
    $\vc{\Omega}_{\tilde{\vc{U}}}[\vc{\Delta}]$.  We conclude that (i)
    $\diff\vc{\Sigma}[\vc{\Delta}]$,
    $\vc{\Omega}_{\vc{V}}[\vc{\Delta}]$ and
    $\vc{\Omega}_{\vc{U}}[\vc{\Delta}]$ are uniquely characterized by
    the SVD of $\vc{X}$, and (ii) they act linearly on
    $\vc{\Delta}$. Furthermore, the differentials
    obey~\eqref{eq:diff_product_svd}. We now prove that
    $\diff\vc{U}[\vc{\Delta}]$ is independent of the choice of
    $\tilde{\vc{U}}$. Decompose
    $\vc{\Omega}_{\tilde{\vc{U}}}[\vc{\Delta}]$ and $\tilde{\vc{U}}$
    as
    \begin{equation*}
        \vc{\Omega}_{\tilde{\vc{U}}}[\vc{\Delta}] = \begin{bmatrix} \vc{\Omega}_{\vc{U}}[\vc{\Delta}] \\ \vc{\Omega}_{\vc{Q}}[\vc{\Delta}]\end{bmatrix}\qquad\text{and}\qquad \tilde{\vc{U}} =
        \begin{bmatrix} \vc{U} & \vc{Q} \end{bmatrix},
    \end{equation*}
    where $\vc{\Omega}_{\vc{U}}[\vc{\Delta}]$ denotes the upper
    $n\times n$ block of $\vc{\Omega}_{\tilde{\vc{U}}}[\vc{\Delta}]$
    (which depends only on $\vc{U}$ by~\eqref{eq:omega_entries}) and
    $\vc{\Omega}_{\vc{Q}}[\vc{\Delta}]$ denotes the lower $(m-n)\times
    n$ block of $\vc{\Omega}_{\tilde{\vc{U}}}[\vc{\Delta}]$ (which
    by~\eqref{eq:omega_eqns_nonsquare} equals
    $\vc{Q}^*\vc{\Delta}\vc{V}\vc{\Sigma}^{-1}$). Then
    \begin{equation*}
        \diff\vc{U}[\vc{\Delta}] = \tilde{\vc{U}}\vc{\Omega}_{\tilde{\vc{U}}}[\vc{\Delta}] = \vc{U}\vc{\Omega}_{\vc{U}}[\vc{\Delta}] + \vc{Q}\vc{\Omega}_{\vc{Q}}[\vc{\Delta}] =
        \vc{U}\vc{\Omega}_{\vc{U}}[\vc{\Delta}] + (\vc{Q}\vc{Q}^*)\vc{\Delta}\vc{V}\vc{\Sigma}^{-1},
    \end{equation*}
    which is well defined since the orthogonal projector
    $\vc{Q}\vc{Q}^*$ is independent of the choice of $\vc{Q}$.
\end{proof}


Since Lemma~\ref{lemma:diff_svd_simple} provides a closed-form
expression for the differentials, computing the value of the
divergence is now a matter of calculus.

\begin{theorem}\label{theorem:div_simple}
  Let $f$ be a matrix-valued spectral function and suppose
  $\vc{X}\in\R^{m\times n}$ is simple and has full-rank. Then
    \begin{equation}\label{eq:theorem:div_simple}
        \div\, (f(\vc{X})) = \sum_{i = 1}^{\min(m, n)} \left(f'_i(\sigma_i) + |m-n|\frac{f'_i(\sigma_i)}{\sigma_i}\right)+ 2\sum_{i\neq j,\, i, j = 1}^{\min(m, n)} \frac{\sigma_i
        f_i(\sigma_i)}{\sigma_i^2 - \sigma_j^2}.
    \end{equation}
\end{theorem}
\begin{proof}
  Assume $m \geq n$ without loss of generality. Let
  $\{\tilde{\vc{u}}_i\}_{i = 1}^{m}$ and $\{\vc{v}_i\}_{i =1}^{n}$
  denote the columns of $\tilde{\vc{U}}$ and $\vc{V}$, and set
    \begin{equation*}
        \vc{\Delta}^{ij} = \tilde{\vc{u}}_i \vc{v}_j^*, \qquad (i,j) \in \idx.
    \end{equation*}
    Then $\{\vc{\Delta}^{ij}\}_{(i,j) \in \idx}$ is an orthonormal
    basis for $\R^{m\times n}$ and, furthermore, $\tilde{\vc{U}}^*
    \vc{\Delta}^{ij} \vc{V} = \vc{E}^{ij}$. Since
    $\{\vc{\Delta}^{ij}\}_{(i,j)\in\idx}$ is an orthonormal basis,
\begin{equation}
\label{eq:div-delta}
  \div\, (f(\vc{X})) = \sum_{(i,j)\in\idx} \iprod{\vc{\Delta}^{ij}}{\diff f[\vc{\Delta}^{ij}]}.
\end{equation}
Note that~\eqref{eq:diff_product_g_rule} yields
    \begin{equation*}
        \diff f[\vc{\Delta}^{ij}] = \tilde{\vc{U}}^*(\vc{\Omega}_{\tilde{\vc{U}}}[\vc{\Delta}^{ij}]f(\vc{\Sigma}) + \id_{m\times n}\diff(f\circ\vc{\Sigma})[\vc{\Delta}^{ij}] +
        \id_{m\times n}
        f(\vc{\Sigma})\vc{\Omega}_\vc{V}[\vc{\Delta}^{ij}]) \vc{V}^*.
    \end{equation*}
    Next, from~\eqref{eq:diff_svalues} we obtain
    \begin{equation*}
      \diff\sigma_k[\vc{\Delta}^{ij}] = \delta_{ki}\delta_{kj}, \qquad (i,j)\in\idx,\quad k = 1,\ldots,n.
    \end{equation*}
    By the same arguments, from~\eqref{eq:omega_entries} we
    obtain
    \begin{equation*}
    	\begin{bmatrix}\Omega_{\tilde{\vc{U}}, kl}[\vc{\Delta}^{ij}] \\ \Omega_{\vc{V}, kl}[\vc{\Delta}^{ij}]\end{bmatrix} = -\frac{1}{\sigma_{k}^2 - \sigma_{l}^2}\begin{bmatrix}\sigma_{l} &
    \sigma_{k} \\ -\sigma_{k} & -\sigma_{l}\end{bmatrix}\begin{bmatrix} \delta_{ik}\delta_{jl} \\  \delta_{il}\delta_{jk}\end{bmatrix},\qquad (i,j)\in\idx,\quad k, l = 1,\ldots, n,\quad k\neq l,
    \end{equation*}
    and from~\eqref{eq:omega_eqns_nonsquare},
    \begin{equation*}
      \Omega_{\tilde{\vc{U}}, kl}[\vc{\Delta}^{ij}] = \frac{\delta_{ik}\delta_{jl}}{\sigma_{l}},\qquad (i,j)\in\idx,\quad k = n+1,\ldots, m,\quad l = 1,\ldots, n.
    \end{equation*}

    Now follow \eqref{eq:div-delta} and decompose the divergence as
    \begin{eqnarray*}
        \div\, (f(\vc{X})) & = & \sum_{(i,j)\in\idx}\iprod{\vc{E}^{ij}}{\vc{\Omega}_{\tilde{\vc{U}}}[\vc{\Delta}^{ij}]f(\vc{\Sigma}) + \id_{m\times n}\diff(f\circ
        \vc{\Sigma})[\vc{\Delta}^{ij}] + \id_{m\times n}f(\vc{\Sigma})\vc{\Omega}_\vc{V}[\vc{\Delta}^{ij}]} \\
            & = & S^{\tilde{\vc{U}}} + S^{\vc{\Sigma}} + S^{\vc{V}}.
    \end{eqnarray*}
    We see that $S^{\vc{V}}$ is equal to
    \begin{equation*}
        S^{\vc{V}} = \sum_{i\neq j,\, i,j = 1}^{n} f_{ij}(\vc{\Sigma})\Omega_{\vc{V}, ij}[\vc{\Delta}^{ij}]
            = \sum_{i\neq j,\, i,j = 1}^{n} \frac{f_i(\sigma_i)}{\sigma_i^2 - \sigma_j^2}(\sigma_i \delta_{ii}\delta_{jj} + \sigma_j \delta_{ij}\delta_{ji})
            = \sum_{i\neq j,\, i,j = 1}^{n} \frac{\sigma_i f_i(\sigma_i)}{\sigma_i^2 - \sigma_j^2},
    \end{equation*}
    whereas
    \begin{equation*}
        S^{\vc{\Sigma}} = \sum_{i,j = 1}^{n} \diff f_{ij}[\diff\Sigma_{ij}[\vc{\Delta}^{ij}]]
            = \sum_{i,j = 1}^{n}  \delta_{ij} f_i'(\sigma_i)
            = \sum_{i = 1}^{n} f_i'(\sigma_i).
    \end{equation*}
    Finally, $S^{\tilde{\vc{U}}}$ decomposes as
    \begin{equation*}
        S^{\tilde{\vc{U}}} = \sum_{i,j = 1}^{n}\iprod{\vc{E}^{ij}}{\vc{\Omega}_{\tilde{\vc{U}}}[\vc{\Delta}^{ij}]f(\vc{\Sigma})} + \sum_{i = n+1}^{m}\sum_{j =
        1}^{n}\iprod{\vc{E}^{ij}}{\vc{\Omega}_{\tilde{\vc{U}}}[\vc{\Delta}^{ij}]f(\vc{\Sigma})}.
    \end{equation*}
    Using~\eqref{eq:omega_entries}, the first term is equal to
    \begin{equation*}
        \sum_{i,j = 1}^{n}\iprod{\vc{E}^{ij}}{\vc{\Omega}_{\tilde{\vc{U}}}[\vc{\Delta}^{ij}]f(\vc{\Sigma})} = -\sum_{i\neq j,\, i,j = 1}^{n} \frac{f_j(\sigma_j)}{\sigma_i^2 -
        \sigma_j^2}(\sigma_j \delta_{ii}\delta_{jj} + \sigma_i \delta_{ij}\delta_{ji})
            = \sum_{i\neq j,\, i,j = 1}^{n} \frac{\sigma_i f_i(\sigma_i)}{\sigma_i^2 - \sigma_j^2},
    \end{equation*}
    while it follows from~\eqref{eq:omega_eqns_nonsquare} that the
    second equals
    \begin{equation*}
        \sum_{i = n+1}^{m}\sum_{j = 1}^{n}\iprod{\vc{E}^{ij}}{\vc{\Omega}_{\tilde{\vc{U}}}[\vc{\Delta}^{ij}]f(\vc{\Sigma})}
            = \sum_{i = 1}^{m - n}\sum_{j = 1}^{n}\iprod{\vc{E}^{ij}}{\vc{E}^{ij}\vc{\Sigma}^{-1}f(\vc{\Sigma})}
            = (m - n) \sum_{i = 1}^{n} \frac{f_i(\sigma_i)}{\sigma_i}.
    \end{equation*}
    Since $n = \min(m, n)$, we conclude that
     \begin{equation*}
        \div\, (f(\vc{X})) = S^{\tilde{\vc{U}}} + S^{\vc{\Sigma}} + S^{\vc{V}} = \sum_{i = 1}^{\min(m, n)} \left(f'_i(\sigma_i) + |m-n|\frac{f'_i(\sigma_i)}{\sigma_i}\right)+ 2\sum_{i\neq
        j,\, i, j = 1}^{\min(m, n)} \frac{\sigma_i f_i(\sigma_i)}{\sigma_i^2 - \sigma_j^2}.
    \end{equation*}
\end{proof}

\subsection{The complex case}

Since the singular values are not analytic functions of the entries,
we can only consider the derivatives of $f$ as a function of the real
and imaginary parts of the entries of $\vc{X}$. Therefore, we identify
$\C^{m\times n}$ with $\R^{m\times n}\times\R^{m\times n}$ and $f$
with a complex-valued function of $2mn$ variables. This has the
particular consequence that the differential $\text{d} f$ at $\vc{X}
\in \C^{m \times n}$---whenever it exists---obeys
\begin{equation*}
  \diff f_\vc{X}[\vc{\Delta}] = \diff f_\vc{X}[\real(\vc{\Delta})] + \diff f_\vc{X}[\1i \imag(\vc{\Delta})],
\end{equation*}
for all $\vc{\Delta}\in\C^{m\times n}$; here, $\1i$ is the imaginary
unit. The differential 
seen as a function with domain $\R^{m\times n}\times\R^{m\times n}$ is
of course linear. Lemma~\ref{lemma:diff_svd_simple} may be adapted
to the complex setting with minor modifications.
\begin{lemma}\label{lemma:diff_svd_simple_cpx}
    Let $\vc{X}$ be simple and full-rank. Then the differentials of
    $\vc{U}$ and $\vc{V}$ are given by
    \[
        \diff\vc{U}[\vc{\Delta}] = \tilde{\vc{U}}
        \vc{\Omega}_{\tilde{\vc{U}}}[\vc{\Delta}], \qquad
        \diff\vc{V}[\vc{\Delta}] = \vc{V} \vc{\Omega}_\vc{V}[\vc{\Delta}]^*,
    \]
    where $\vc{\Omega}_{\tilde{\vc{U}}}$ and $\vc{\Omega}_\vc{V}$ are
    given in closed-form \eqref{eq:omega_entries_cpx} ($\diff\vc{U}$ is
    independent of the choice of $\tilde{\vc{U}}$). Further,
    $\diff\vc{\Sigma}$ is given by \eqref{eq:diff_svalues_cpx}.
    For completeness, $\vc{\Omega}_{\tilde{\vc{U}}},
    \vc{\Omega}_\vc{V}: \C^{m\times n}\mapsto \C^{m\times n}$ while
    $\diff\vc{\Sigma}:\C^{m\times n}\mapsto \R^{m\times n}$.
\end{lemma}
\begin{proof}
    Following the same steps as in the proof of
    Lemma~\ref{lemma:diff_svd_simple}, let $\vc{X}\in\C^{m\times n}$ ($m \geq
    n$). We see that~\eqref{eq:diff_product_svd_omega} becomes
    \begin{equation}\label{eq:diff_product_svd_omega_cpx}
        \tilde{\vc{U}}^*\vc{\Delta} \vc{V} = \vc{\Omega}_{\tilde{\vc{U}}}[\vc{\Delta}]\vc{\Sigma} + \id_{m\times n}\diff\vc{\Sigma}[\vc{\Delta}] + \id_{m\times n}\vc{\Sigma}
        \vc{\Omega}_\vc{V}[\vc{\Delta}],
    \end{equation}
    where $\vc{\Omega}_\vc{V}$ and the upper $n\times n$ block of
    $\vc{\Omega}_{\tilde{\vc{U}}}$ are now skew-Hermitian. Since
    these matrices have purely imaginary entries on the diagonal,
    \begin{equation}\label{eq:diff_svalues_cpx}
    	\diff\sigma_i[\vc{\Delta}] = \real\{(\tilde{\vc{U}}^* \vc{\Delta} \vc{V})_{ii}\} = \frac{(\tilde{\vc{U}}^* \vc{\Delta} \vc{V})_{ii} + \conj{(\tilde{\vc{U}}^* \vc{\Delta}
    \vc{V})_{ii}}}{2},\qquad i = 1,\ldots, n,
    \end{equation}
    and
    \begin{equation}\label{eq:omega_entries_diag_cpx}
        \Omega_{\tilde{\vc{U}}, ii}[\vc{\Delta}]\sigma_i + \Omega_{\vc{V}, ii}[\vc{\Delta}]\sigma_i = \imag\{(\tilde{\vc{U}}^* \vc{\Delta} \vc{V})_{ii}\}\1i = \frac{(\tilde{\vc{U}}^* \vc{\Delta}
        \vc{V})_{ii} - \conj{(\tilde{\vc{U}}^* \vc{\Delta} \vc{V})_{ii}}}{2},\qquad i = 1,\ldots, n.
    \end{equation}
    Furthermore, it is not hard to see that the
    system~\eqref{eq:omega_entries} becomes
    \begin{equation}\label{eq:omega_entries_cpx}
    	\begin{bmatrix}\Omega_{\tilde{\vc{U}}, ij}[\vc{\Delta}] \\ \Omega_{\vc{V}, ij}[\vc{\Delta}]\end{bmatrix} = -\frac{1}{\sigma_{i}^2 - \sigma_{j}^2}\begin{bmatrix}\sigma_{j} & \sigma_{i} \\
    -\sigma_{i} & -\sigma_{j}\end{bmatrix}\begin{bmatrix} (\vc{U}^*\vc{\Delta} \vc{V})_{ij} \\ \conj{(\vc{U}^*\vc{\Delta} \vc{V})_{ji}}\end{bmatrix},\qquad i\neq j,\quad i,j = 1,\ldots,n,
    \end{equation}
    and the expression equivalent to~\eqref{eq:omega_eqns_nonsquare}
    is
    \begin{equation}\label{eq:omega_eqns_nonsquare_cpx}
      (\tilde{\vc{U}}^T\vc{\Delta} \vc{V})_{ij} = \Omega_{\tilde{\vc{U}}, ij}[\vc{\Delta}]\sigma_{j},\qquad i = n + 1,\ldots, m,\quad  j = 1,\ldots, n.
    \end{equation}
    As in the real case, this shows $\diff\vc{\Sigma}[\vc{\Delta}]$,
    $\vc{\Omega}_{\vc{V}}[\vc{\Delta}]$ and
    $\vc{\Omega}_{\vc{U}}[\vc{\Delta}]$ act linearly on $\vc{\Delta}$,
    and that they are uniquely characterized\footnote{The attentive
      reader will notice that~\eqref{eq:omega_entries_diag_cpx} leaves
      us with one degree of freedom, and thus does not determine {\it
        exactly} the differentials. This is not a problem, and it is a
      consequence of the freedom of choice of a phase factor for the
      singular vectors.} by the SVD of $\vc{X}$. Furthermore, the same
    argument allows us to show that the differential of $\vc{U}$ is
    independent of the choice of $\tilde{\vc{U}}$, concluding the proof.
%
\end{proof}

As in the real case, we restrict attention to a simple and full-rank
$\vc{X}\in\C^{m\times n}$. We need to determine two terms for the
divergence as discussed in Section~\ref{section:sure_for_svt:cpx}. Now
if $\{\vc{\Delta}^{ij}\}_{(i,j)\in\idx}$ is an orthonormal basis for
$\R^{m\times n}$, then $\{\1i\vc{\Delta}^{ij}\}_{(i,j)\in\idx}$ is an
orthonormal basis for $\1i\R^{m\times n}$. With the divergence as
in~\eqref{eq:def_div_cpx},
\begin{eqnarray*}
    \div\, (f(\vc{X})) & = & \sum_{(i,j)\in\idx} \real(\iprod{\vc{\Delta}^{ij}}{\diff f[\vc{\Delta}^{ij}]}) + \sum_{(i,j)\in\idx} \imag(\iprod{\vc{\Delta}^{ij}}{\diff f[\1i\vc{\Delta}^{ij}]}) \\
        & = & \real\left(\sum_{(i,j)\in\idx} \iprod{\vc{\Delta}^{ij}}{\diff f[\vc{\Delta}^{ij}] - \1i \diff f[\1i\vc{\Delta}^{ij}]} \right).
\end{eqnarray*}
The analogue of Theorem~\ref{theorem:div_simple} is this:
\begin{theorem}\label{theorem:div_simple_cpx}
 Let $f$ be a matrix-valued spectral function and suppose
  $\vc{X}\in\C^{m\times n}$ is simple and has full-rank. Then
    \begin{equation}\label{eq:theorem:div_simple_cpx}
        \div\, (f(\vc{X})) = \sum_{i = 1}^{\min(m, n)}\left(f_i'(\sigma_i) + (2|m-n|+1)\frac{f_i(\sigma_i)}{\sigma_i}\right) + 4\sum_{i\neq j,\, i,j = 1}^{\min(m, n)}\frac{\sigma_i
        f_i(\sigma_i)}{\sigma_i^2 - \sigma_j^2}.
    \end{equation}
\end{theorem}
\begin{proof}
  Since $\vc{X}$ is simple and has full-rank,
  Lemma~\ref{lemma:diff_svd_simple_cpx} holds, and
    \begin{multline}
      \diff f[\vc{E}^{ij}] - \1i \diff f[\vc{E}^{ij}]  =
      \tilde{\vc{U}}\left[(\vc{\Omega}_{\tilde{\vc{U}}}[\vc{E}^{ij}]-\1i\vc{\Omega}_{\tilde{\vc{U}}}[\1i\vc{E}^{ij}])f(\vc{\Sigma})
        +
        \id_{m\times n}f(\vc{\Sigma})(\vc{\Omega}_\vc{V}[\vc{E}^{ij}] - \1i\vc{\Omega}_\vc{V}[\1i\vc{E}^{ij}])\right] \vc{V}^*  \\
       +  \tilde{\vc{U}}\id_{m\times
        n}\diff f_{\vc{\Sigma}}[\diff\vc{\Sigma}[\vc{E}^{ij}] -
      \1i\diff\vc{\Sigma}[\1i\vc{E}^{ij}]]\vc{V}^*.  \label{eq:df_for_complex_div}
    \end{multline}
    Put $\vc{L}^{ij} = \tilde{\vc{U}}^* \vc{E}^{ij} \vc{V}$ for
    $(i,j)\in \idx$ for short. From~\eqref{eq:diff_svalues_cpx}, we
    get
    \begin{equation*}
    	\diff\sigma_i[\vc{E}^{ij}] - \1i\diff\sigma_i[\1i\vc{E}^{ij}] = L_{kk}^{ij}
    \end{equation*}
    for $(i,j)\in\idx$ and $k = 1,\ldots, n$. By the same arguments,~\eqref{eq:omega_entries_cpx} gives
    \begin{equation*}
    	\begin{bmatrix}\Omega_{\tilde{\vc{U}}, kl}[\vc{E}^{ij}] - \1i\Omega_{\tilde{\vc{U}}, kl}[\1i \vc{E}^{ij}] \\ \Omega_{\vc{V}, kl}[\vc{E}^{ij}] - \1i \Omega_{\vc{V},
    kl}[\vc{E}^{ij}]\end{bmatrix} = -\frac{1}{\sigma_{k}^2 - \sigma_{l}^2}\begin{bmatrix}\sigma_{l} & \sigma_{k} \\
    -\sigma_{k} & -\sigma_{l}\end{bmatrix}\begin{bmatrix}2 L^{ij}_{kl} \\  0\end{bmatrix},
    \end{equation*}
    for $(i,j)\in\idx$, $k, l = 1,\ldots, n$, and $k\neq l$. Next,~\eqref{eq:omega_eqns_nonsquare_cpx} gives
    \begin{equation*}
        \Omega_{\tilde{\vc{U}}, kl}[\vc{E}^{ij}] = 2\frac{L^{ij}_{kl}}{\sigma_{l}},
    \end{equation*}
    for $(i,j)\in\idx$, $k = n+1,\ldots, m$, and $l = 1,\ldots, n$. Finally,~\eqref{eq:omega_entries_diag_cpx} implies
    \begin{equation*}
        \vc{\Omega}_{\tilde{\vc{U}},kk}[\vc{E}^{ij}] + \Omega_{\vc{V}, kk}[\vc{E}^{ij}] - \1i (\Omega_{\tilde{\vc{U}}, kk}[\1i \vc{E}^{ij}] + \Omega_{\vc{V}, kk}[\1i \vc{E}^{ij}]) =
        \frac{L^{ij}_{kk}}{\sigma_k},
    \end{equation*}
    for $k = 1,\ldots, n$. Now, set
    \begin{eqnarray*}
        S^{\tilde{\vc{U}}}  & = &   \sum_{(i,j)\in\idx} \iprod{\vc{L}^{ij}}{(\vc{\Omega}_{\tilde{\vc{U}}}[\vc{E}^{ij}]-\1i\vc{\Omega}_{\tilde{\vc{U}}}[\1i\vc{E}^{ij}])f(\vc{\Sigma})},\\
        S^{\vc{\Sigma}}     & = &   \sum_{(i,j)\in\idx} \iprod{\vc{L}^{ij}}{\id_{m\times n}f(\vc{\Sigma})(\vc{\Omega}_\vc{V}[\vc{E}^{ij}] - \1i\vc{\Omega}_\vc{V}[\1i\vc{E}^{ij}])},\\
        S^{\vc{V}}          & = &   \sum_{(i,j)\in\idx} \iprod{\vc{L}^{ij}}{\id_{m\times n}\diff f_{\vc{\Sigma}}[\diff\vc{\Sigma}[\vc{E}^{ij}] - \1i\diff\vc{\Sigma}[\1i\vc{E}^{ij}]]}.
    \end{eqnarray*}
    Then it is clear that
    \begin{eqnarray*}
      S^{\vc{V}} & = & \sum_{(i,j)\in\idx} \sum_{k\neq l,\, k,l = 1}^{n} \frac{2 f_k(\sigma_k) \sigma_k |L^{ij}_{kl}|^2}{\sigma_k^2 - \sigma_l^2} + \sum_{k =
        1}^{n}f_k(\sigma_k)\sum_{(i,j)\in\idx} \conj{L^{ij}_{kk}} (\vc{\Omega}_{\vc{V},kk}[\vc{E}^{ij}] - \1i \vc{\Omega}_{\vc{V}, kk}[\1i\vc{E}^{ij}]) \\
      & = & 2\sum_{k\neq l,\, k,l = 1}^{n} \frac{f_k(\sigma_k) \sigma_k}{\sigma_k^2 - \sigma_l^2} + \sum_{k = 1}^{n}f_k(\sigma_k)\sum_{(i,j)\in\idx} \conj{L^{ij}_{kk}}
      (\vc{\Omega}_{\vc{V},kk}[\vc{E}^{ij}] - \1i \vc{\Omega}_{\vc{V}, kk}[\1i\vc{E}^{ij}]),
    \end{eqnarray*}
    and
    \begin{equation*}
        S^{\vc{\Sigma}} = \sum_{(i,j)\in\idx} \sum_{k = 1}^{n} f_k'(\sigma_k) |L_{kk}^{ij}|^2 =  \sum_{k = 1}^{n} f_k'(\sigma_k).
    \end{equation*}
    Finally,
    \begin{align*}
      S^{\tilde{\vc{U}}} & = -\sum_{(i,j)\in\idx}\sum_{k\neq l,\, k,l
        = 1}^{n} \frac{2f_l(\sigma_l)\sigma_l
        |L_{kl}^{ij}|^2}{\sigma_k^2 - \sigma_l^2} + \sum_{k =
        1}^{n}f_k(\sigma_k)\sum_{(i,j)\in\idx} \conj{L^{ij}_{kk}} (\vc{\Omega}_{\tilde{\vc{U}},kk}[\vc{E}^{ij}] - \1i \vc{\Omega}_{\tilde{\vc{U}}, kk}[\1i\vc{E}^{ij}]) \\
      & \qquad \qquad +    \sum_{(i,j)\in\idx}\sum_{k = n+1}^{m}\sum_{l = 1}^{n} 2\frac{f_l(\sigma_l) |L^{ij}_{kl}|^2}{\sigma_{l}} \\
      & = 2\sum_{k\neq l,\, k,l = 1}^{n}
      \frac{f_k(\sigma_k)\sigma_k}{\sigma_k^2 - \sigma_l^2} + 2(m - n)
      \sum_{k = 1}^{n} \frac{f_k(\sigma_k)}{\sigma_{k}} + \sum_{k =
        1}^{n}f_k(\sigma_k)\sum_{(i,j)\in\idx} \conj{L^{ij}_{kk}}
      (\vc{\Omega}_{\tilde{\vc{U}},kk}[\vc{E}^{ij}] - \1i
      \vc{\Omega}_{\tilde{\vc{U}}, kk}[\1i\vc{E}^{ij}]).
    \end{align*}
    Since
    \begin{equation*}
        \sum_{(i,j)\in\idx} \conj{L^{ij}_{kk}} \left(\vc{\Omega}_{\tilde{\vc{U}},kk}[\vc{E}^{ij}] - \1i \vc{\Omega}_{\tilde{\vc{U}}, kk}[\1i\vc{E}^{ij}] + \vc{\Omega}_{\vc{V},kk}[\vc{E}^{ij}] -
        \1i \vc{\Omega}_{\vc{V}, kk}[\1i\vc{E}^{ij}]\right) = \sum_{(i,j)\in\idx} \frac{|L^{ij}_{kk}|^2}{\sigma_k} = \frac{1}{\sigma_k},
    \end{equation*}
    we conclude
    \begin{eqnarray*}
        \div\, (f(\vc{X})) & = & \real\left(S^{\tilde{\vc{U}}} + S^{\vc{\Sigma}} + S^{\vc{V}}\right) \\
            & = & \sum_{i = 1}^{\min(m, n)}\left(f_i'(\sigma_i) + (2|m-n|+1)\frac{f_i(\sigma_i)}{\sigma_i}\right) + 4\sum_{i\neq j,\, i,j = 1}^{\min(m, n)}\frac{\sigma_i
        f_i(\sigma_i)}{\sigma_i^2 - \sigma_j^2},
    \end{eqnarray*}
    where we used $n = \min(m, n)$.
\end{proof}

\subsection{An extension to all matrix arguments}


Theorems~\ref{theorem:div_simple} and~\ref{theorem:div_simple_cpx}
provide the divergence of a matrix-valued spectral function in
closed-form for simple, full-rank arguments. Here, we indicate that
these formulas have a continuous extension to all matrices.  Before
continuing, we introduce some notation.  For a given $\vc{X}$, we
denote as $\kappa \leq n$ the number of distinct singular values while
$s_i$ and $d_i$, $i = 1, \ldots, \kappa$, denote the $i$-th distinct
singular value and its multiplicity.

As previously discussed, when $\vc{\Sigma}$ is simple there is no
ambiguity as to which $f_i$ is being applied to a singular
value. However, when $\vc{\Sigma}$ is not simple, the ambiguities make
$f$ nondifferentiable unless $f_i \equiv f$ for $i = 1,\ldots, n$.
\begin{theorem}\label{theorem:div_all}
  Let $f$ be a matrix-valued spectral function, with $f_i \equiv f$
  for $i = 1,\ldots, n$, and $f(0) = 0$.  Then the divergence extends
  by continuity to any $\vc{X}\in\R^{m\times n}$, and is given by
    \begin{eqnarray}
        \div\, (f(\vc{X})) & = & \sum_{i:\, s_i > 0} \left[\left(d_i + {d_i \choose 2}\right)f'(s_i) + \left(|m-n|d_i + {d_i\choose 2}\right)\frac{f(s_i)}{s_i}\right] \nonumber\\
            & + & \sum_{i:\, s_i = 0} \left[(|m - n|+1)d_i + 2{d_i\choose 2}\right]f'(0) + 2\sum_{i\neq j,\, i, j = 1}^{\kappa} d_i d_j\frac{s_i f(s_i)}{s_i^2 - s_j^2},\label{eq:theorem:div_all}
    \end{eqnarray}
    and to any $\vc{X}\in\C^{m\times n}$,
    \begin{eqnarray}
        \div\, (f(\vc{X})) & = & \sum_{i:\, s_i > 0} \left[\left(d_i + 2{d_i \choose 2}\right)f'(s_i) + \left((2|m-n| + 1)d_i + 2{d_i\choose 2}\right)\frac{f(s_i)}{s_i}\right] \nonumber\\
            & + & \sum_{i:\, s_i = 0} \left[2(|m - n| + 1)d_i + 4{d_i\choose 2}\right]f'(0) + 4\sum_{i\neq j,\, i, j = 1}^{\kappa} d_i d_j\frac{s_i f(s_i)}{s_i^2 - s_j^2}.\label{eq:theorem:div_all_cpx}
    \end{eqnarray}
\end{theorem}

\begin{proof}
  We only prove \eqref{eq:theorem:div_all} as the same arguments apply
  to \eqref{eq:theorem:div_all_cpx}. The key idea involves
  non-tangential limits defined below.  Suppose we have a sequence
  $\{\vc{X}_k\}$ of full-rank matrices with $\vc{X}_k\rightarrow
  \vc{X}$ as $k\rightarrow\infty$; we establish that the limit
  $\div(f(\vc{X}_k))$ is independent of the
  sequence. Since~\eqref{eq:theorem:div_simple} depends only on the
  singular values of $\vc{X}_k$, we can focus on the sequence of
  singular values.  Let $\{\sigma_i\}_{i = 1}^{n}$ be the singular
  values of $\vc{X}$, and let $\{\sigma_i^k\}_{i = 1}^{n}$ be those of
  $\vc{X}_k$. We have
    \begin{equation*}
        \div\, (f(\vc{X}_k)) = \sum_{i = 1}^{\min(m, n)} \left(f'(\sigma_i^k) + |m-n|\frac{f(\sigma_i^k)}{\sigma_i^k}\right)+ 2\sum_{i\neq j,\, i, j = 1}^{\min(m, n)} \frac{\sigma_i^k f(\sigma_i^k)}{(\sigma_i^k)^2 - (\sigma_j^k)^2}.
    \end{equation*}
    The first terms are easy to analyze. In fact,
    \begin{equation*}
        \lim_{k\rightarrow\infty} f'(\sigma_i^k) = f'(\sigma_i)\quad\text{and}\quad \lim_{k\rightarrow\infty}\frac{f(\sigma_i^k)}{\sigma_i^k} = \begin{cases}\frac{f(\sigma_i)}{\sigma_i}, & \text{if
$\sigma_i > 0$,}\\ f'(0), & \text{if $\sigma_i = 0$.}\end{cases}
    \end{equation*}
    The remaining term can be decomposed as
    \begin{equation}\label{eq:div_all_non_tangential}
      \sum_{i\neq j,\, i, j = 1}^{\min(m, n)} \frac{\sigma_i^k f(\sigma_i^k)}{(\sigma_i^k)^2 - (\sigma_j^k)^2} = \left[\sum_{i\neq j,\, \sigma_i = \sigma_j}^{\min(m, n)} + \sum_{i\neq j,\, \sigma_i \neq \sigma_j}^{\min(m, n)}\right] \frac{\sigma_i^k f(\sigma_i^k)}{(\sigma_i^k)^2 - (\sigma_j^k)^2} := I_0 + I_1.
    \end{equation}

    We now formalize the notion of non-tangential limit. Consider a
    pair $(i,j)$ with $i \neq j$ and $\sigma_i = \sigma_j := \sigma$,
    and define
    \begin{equation*}
      \delta^{k}_i = \sigma_i^k - \sigma,
      \quad \delta^{k}_j = \sigma_j^{k}  - \sigma
      \quad\text{and}\quad \rho^{k}_{ij} = \delta^k_{j}/\delta^k_{i}.
    \end{equation*}
    We say that $\vc{X}_k$ approaches $\vc{X}$ non-tangentially if
    there exists a constant $c_{ij} > 0$ such that $|\rho^k_{ij} - 1|
    > c_{ij}$. Roughly speaking, the matrices $\vc{X}_k$ approach
    $\vc{X}$ fast enough, so that they remain simple, and have
    full-rank.  We have
    \begin{eqnarray*}
        \frac{\sigma_i^k f(\sigma_i^k) - \sigma_j^k f(\sigma_j^k)}{(\sigma_i^k)^2 - (\sigma_j^k)^2} & = & \frac{\sigma}{2\sigma + \delta^k_i + \delta^k_j}\left(\frac{f(\sigma + \delta_i^k) - f(\sigma + \delta_j^k)}{\delta_i^k - \delta_j^k}\right) \\
            & + & \frac{1}{2\sigma + \delta^k_i + \delta^k_j}\left(\frac{1}{1 - \rho_{ij}^k}f(\sigma + \delta_i^k) - \frac{\rho^k_{ij}}{1 - \rho_{ij}^k}f(\sigma + \delta_j^k)\right).
    \end{eqnarray*}
    We distinguish two cases. If $\sigma > 0$, we see that
    \begin{equation*}
        \lim_{k\rightarrow\infty}\frac{\sigma_i^k f(\sigma_i^k) - \sigma_j^k f(\sigma_j^k)}{(\sigma_i^k)^2 - (\sigma_j^k)^2} = \frac{1}{2} f'(\sigma) + \frac{1}{2}\frac{f(\sigma)}{\sigma}
    \end{equation*}
    whereas when $\sigma = 0$ we obtain
    \begin{equation*}
      \lim_{k\rightarrow\infty}\frac{\sigma_i^k f(\sigma_i^k) - \sigma_j^k f(\sigma_j^k)}{(\sigma_i^k)^2 - (\sigma_j^k)^2} = f'(0)
    \end{equation*}
    and, therefore,
    \begin{equation*}
      I_0 = \sum_{i:\, s_i = 0} {d_i \choose 2}f'(0) + \frac{1}{2}\sum_{i:\, s_i > 0} {d_i \choose 2}\left(f'(s_i) + \frac{f(s_i)}{s_i}\right).
    \end{equation*}
    The second sum $I_1$ in~\eqref{eq:div_all_non_tangential} is
    easier to analyze, as it can be seen that
    \begin{equation*}
        \lim_{k\rightarrow\infty}\sum_{i\neq j,\, \sigma_i \neq \sigma_j}^{\min(m, n)}\frac{\sigma_i^k
        f(\sigma_i^k)}{(\sigma_i^k)^2 - (\sigma_j^k)^2} = \sum_{i\neq j,\, i,j = 1}^{\kappa} d_i d_j\frac{s_i
        f(s_i)}{s_i^2 - s_j^2},
    \end{equation*}
    concluding the proof.
\end{proof}

\section{Discussion}
\label{section:discussion}

Beyond MRI, SVT denoising also has potential application for medical
imaging modalities outside of MRI. For example, in perfusion X-ray
computed tomography (CT) studies, where some anatomical region (e.g.,
the kidneys) is repeatedly scanned to visualize hemodynamics, X-ray
tube energy is routinely lowered to limit the ionizing radiation dose
delivered to the patient. However, this can result in a substantial
increase of image noise. Recent works (e.g., \cite{Borsdorf2008}) have
shown that retrospective denoising of low-dose CT data can often yield
diagnostic information comparable to full-dose images. It may be
possible to adapt SVT, and the proposed risk estimators, for the CT
dose reduction problem as well.

Following the successful application of compressive sensing
theory~\cite{Candes2006,Donoho2006} for accelerated MRI (e.g.,
~\cite{Lustig2007,Trzasko2009}), the use of low-rank constraints for
undersampled MRI reconstruction is also of increasing interest. For
the dynamic and parametric acquisition scenarios described above, the
generalization from denoising to undersampled reconstruction is
straightforward in light of recent developments on the topic of matrix
completion~\cite{CR08,Recht2010}. To date, low-rank reconstruction
methods have been deployed for accelerated cardiac
imaging~\cite{Haldar2010,Lingala2011,Trzasko2011} and dynamic
contrast-enhanced MRI (DCE-MRI) of the breast~\cite{Haldar2010}. More
recently, it has been demonstrated that reconstructing accelerated
parallel MRI data under low-rank
constraints~\cite{Lustig2010,Trzasko2012b} can obviate the need for
expensive and potentially error-prone patient-specific calibration
procedures as well as facilitate previously impractical acquisition
strategies. As with the denoising application, parameter selection is
an outstanding challenge in undersampled reconstruction
problems. Unlike the denoising problem, only incomplete noisy data is
available and Stein's method is not directly applicable. Recently, for
undersampled MRI reconstruction, Ramani et al.~\cite{Ramani2012}
demonstrated that Stein's method can be used to generate an unbiased
risk estimator for the predicted MSE of $\ell_{1}$-minimization
reconstructions, and that this can be used for automated parameter
optimization. An interesting future direction of research will lie in
determining if the models developed in this work facilitate
development of a similar framework for the low-rank matrix recovery
problem.

\small
\subsection*{Acknowledgements}
E.~C.~is partially supported by AFOSR under grant FA9550-09-1-0643, by
ONR under grant N00014-09-1-0258 and by a gift from the Broadcom
Foundation.  C.~S-L.~is partially supported by a Fulbright-CONICYT
Scholarship and by the Simons Foundation. J.~D.~T.~and
the Mayo Clinic Center for Advanced Imaging Research are partially
supported by National Institute of Health grant RR018898.
E.~C.~would like to thank Rahul Mazumder for fruitful discussions about this
project.


\bibliographystyle{plain}
\bibliography{paper}

%
%

\end{document}